\documentclass{article}

\usepackage{PRIMEarxiv}

\usepackage[utf8]{inputenc} 
\usepackage[T1]{fontenc}    
\usepackage{hyperref}       
\usepackage{url}            
\usepackage{booktabs}       
\usepackage{amsfonts}       
\usepackage{amsmath}
\usepackage{amssymb}
\usepackage{amsthm}
\usepackage{nicefrac}       
\usepackage{microtype}      
\usepackage{lipsum}
\usepackage{fancyhdr}       
\usepackage{graphicx}       
\graphicspath{{media/}}     

\theoremstyle{plain}
\newtheorem{thm}{Theorem}

\newtheorem{prop}[thm]{Proposition}
\newtheorem{lemma}[thm]{Lemma}

\newtheorem{rem}[thm]{Remark}

\theoremstyle{definition}
\newtheorem{defn}[thm]{Definition}

\numberwithin{equation}{section}
\numberwithin{thm}{section}
  
\title{Properties of the continuum seed-bank coalescent
}

\author{
 Likai Jiao \\
  Humboldt-Universit\"{a}t zu Berlin\\
  \texttt{likai.jiao@hu-berlin.de} \\
}

\begin{document}
\maketitle

\begin{abstract}
We study properties of the \textit{continuum seed-bank coalescent} proposed by \cite{jiao2023wright} such as \textit{not coming down from infinity} and \textit{bounds on the expected time to the most recent common ancestor}. We also provide results on the limiting distributions and propose
some interesting open problems.  
	
\end{abstract}

\section{Introduction}
This paper is devoted to the study of properties for the \textit{continuum seed-bank coalescent} proposed in \cite{jiao2023wright}, such as \textit{not coming down from infinity} and \textit{bounds on the expected time to the most recent common ancestor (MRCA)}, which serves as a further generalization of the relevant results presented in \cite{blath2016}. \textit{Not coming down from infinity} means that the initial infinitely many blocks will not coalesce into finitely many ones at any given time, and \textit{bounds on the expected time to the MRCA} refers to estimations on the expected time it takes for all initial blocks to coalesce into one. As byproducts, we also give a graphical construction of the \textit{continuum seed-bank coalescent}, and find out the limiting distribution of its dual process which is known as the \textit{continuum seed-bank diffusion}. 

The concept of \textit{seed-bank coalescent} initiated by \cite{blath2016} is an extension of the Kingman coalescent (\cite{kingman1982coalescent}, \cite{kingman1982genealogy}) by incorporating a \textit{dormancy} mechanism. During the coalescent process, some blocks of the partition may enter dormant 
states in which they do not participate in the coalescence. As a phenomenon prevalent in the biosphere, \textit{dormancy} has been recognized in recent years as an important evolutionary force in the population evolution (see e.g. \cite{lennon2011microbial}, \cite{shoemaker2018evolution}, \cite{lennon2021principles}). In the aspect of mathematical modeling, since \cite{blath2016}, this dormancy mechanism has been incorporated into various extensions of the Kingman coalescent (see e.g. \cite{blath2020seed}, \cite{greven2022spatial}). 

For the original \textit{seed-bank coalescent}, the dormancy time, i.e., how long 
a dormancy period lasts, follows an exponential distribution since the transition rate $\lambda$ from \textit{dormant} state to \textit{active} sate is a constant. In order to allow for more general dormancy time distributions while preserving the Markov property, \cite{greven2022spatial} proposed the method of considering a countable number of \textit{colored} seed-banks: when a block goes into dormancy, it randomly chooses a \textit{color}, which represents its transition rate $\lambda$ for reviving. As the result, the cumulative distribution function (CDF) of the dormancy time becomes a linear combination of exponential distribution functions, which can approximate a heavy-tailed distribution by tuning the parameters appropriately. This idea is essentially a randomization of the transition rate $\lambda$, inspired by which \cite{jiao2023wright} proposed the more general \textit{continuum seed-bank coalescent} as the limit.

The term \textit{continuum} specifically refers to possibly being \textit{continuous in state}. For a sample composed of finite individuals, the state space of the seed-bank coalescent in \cite{greven2022spatial} is discrete since the distribution of $\lambda$ is discrete, i.e., weighted sum of Dirac measures. However, for the \textit{continuum seed-bank coalescent}, we can take any probability measure $\nu$ on $(0,\infty)$, especially those with support having a continuum cardinality, as the distribution of $\lambda$. Then, the transition rate from \textit{active} state to \textit{dormant} state is described by the measure $\mu:=c\nu$, where $c>0$ is some prescribed total transition rate. For example, taking $\nu$ to be the Gamma distribution $\Gamma(a,b)$ with parameters $a>0, b>0$, the dormancy time has CDF $K(t)=\int_{(0,\infty)}(1-e^{-\lambda t})d\nu=1-\frac{1}{(1+\frac{t}{b})^{a}}, t\geq 0$, i.e., it follows a type of Pareto distribution (Lomax distribution). Obviously, not all CDF can be represented in this way, the scope of applicability is indicated by Bernstein-Widder theorem (\cite{bernstein1929fonctions}): every completely monotone survival function is the Laplace transform of some probability measure $\nu$.

In the spirit of \cite{blath2016}, we will examine certain properties of the generalized \textit{continuum seed-bank coalescent}. The main results are summarized as follows:
\begin{itemize}
	\item \textbf{Not coming down from infinity} 

This is naturally to be expected since the \textit{seed-bank coalescent} (single seed-bank case) does not \textit{come down from infinity} (see Theorem 4.1 in \cite{blath2016}), and now we have more seed-banks. The notion of \textit{come down from infinity} was introduced by \cite{pitman1999coalescents}. For the \textit{continuum seed-bank coalescent}, it implies that $P(N_{t}+||M_{t}||_{TV})<\infty \textit{ for any } t>0)=1$, where $\{(N_{t}, M_{t})\}_{t\geq 0}$ is the process of counting the number of active/dormant blocks in partitions (see Definition \ref{defn:block}), the second component $\{M_{t}\}_{t\geq 0}$ takes finite integer-valued measures as its values, and $||\cdot||_{TV}$ is the total variation norm. But actually the total number of blocks always stays infinite.
\begin{thm}\label{th:coming}
If $N_{0}+||M_{0}||_{TV}$ is countably infinite, then 
\begin{equation}\label{eq:coming}
	P(N_{t}+||M_{t}||_{TV}=\infty, \text{ for any } t\geq 0)=1.
\end{equation}
\end{thm}

\item \textbf{Bounds on the expected time to the most recent common ancestor}

The time to the most recent common ancestor ($T_{MRCA}$) is the distance from the leaf nodes to the root node (MRCA) on the coalescent tree, which is defined by
\begin{equation}
T_{MRCA}(n,m):=\inf\{t>0: (N_{t}^{(n,m)}, M_{t}^{(n,m)})=(1,0)\},	
\end{equation}
where the superscript $(n, m)$ indicates the initial value $(N_{0}, M_{0})=(n, m)$.

Unlike the single seed-bank case, the expectation $E[T_{MRCA}(n,m)]$ of a \textit{continuum seed-bank coalescent} may be infinite. 

\begin{thm}\label{th:mean}
	$E[T_{MRCA}(n,m)]<\infty$ if and only if the expected dormancy time $\int_{(0,\infty)}\frac{1}{\lambda}\mu(d\lambda)<\infty$. 
\end{thm}
Under the assumption that the transition rates of initial dormant individuals are independent $\nu$-distributed random variables, we have the following result providing asymptotic bounds of $E[T_{MRCA}(n,m)]$. The proof employs the same approach as in \cite{blath2016} with slight
modifications, but performs a more detailed calculation of the asymptotic coefficients. They are $\frac{1}{\overline{\lambda}_{0}}$ and $\frac{2}{\underline{\lambda}_{0}}$ for the lower and the upper bound, respectively.

\begin{thm}\label{th:bounds}~
\begin{enumerate}
	\item If the support of $\mu$ has an upper bound $\overline{\lambda}_{0}>0$,
then \begin{equation}
\liminf_{\substack{n\rightarrow \infty\\ ||m||_{TV}\rightarrow\infty}}\frac{E[T_{MRCA}(n,m)]}{\log(\log n+\frac{||m||_{TV}}{2c})}\geq \frac{1}{\overline{\lambda}_{0}}.	
\end{equation}
If $\mu$ is the Gamma distribution $\Gamma(a, b), a>0, b>0$, then
\begin{equation}
\liminf_{\substack{n\rightarrow \infty\\ ||m||_{TV}\rightarrow\infty}}\frac{E[T_{MRCA}(n,m)]}{\log(\log n+\frac{||m||_{TV}}{2c})}\geq \frac{b}{a}.	
\end{equation}
\item If the support of $\mu$ has an upper bound $\overline{\lambda}_{0}>0$ and a lower bound $\underline{\lambda}_{0}>0$,
then \begin{equation}
\limsup_{\substack{n\rightarrow \infty\\ ||m||_{TV}\rightarrow\infty}}\frac{E[T_{MRCA}(n,m)]}{\log(\log n+\frac{||m||_{TV}}{2c})}\leq \frac{2}{\underline{\lambda}_{0}}.	
\end{equation}
\end{enumerate}
\end{thm}
The asymptotic upper bound of $E[T_{MRCA}(n,m)]$ for the Gamma distribution remains an open problem.
\end{itemize}

The remainder of this paper is organized as follows: In Section 2, we review the definition of the \textit{continuum seed-bank coalescent}, then we discuss some basic properties such as duality, exchangeability, limiting distribution of the ancestral line, and comparisons of $E[T_{MRCA}]$. As byproducts, we give a graphical construction of the coalescent in Subsection 2.2, prove Theorem \ref{th:mean} and find out the limiting distribution of the \textit{continuum seed-bank diffusion} in Subsection 2.3 (see Proposition \ref{prop:fixation}). We also propose some interesting open problems in Subsection 2.4 (see Remark \ref{rem:open}). In Section 3 and Section 4, we introduced the \textit{accelerated coalescent} and the \textit{decelerated coalescent}, respectively, following the same approach as in\cite{blath2016}. Then, we prove Theorem \ref{th:coming} and Theorem \ref{th:bounds}-1 in Section 3, and prove Theorem \ref{th:bounds}-2 in Section 4.
\section{Definition and basic properties}

In this section, we review the definition of the \textit{continuum seed-bank coalescent} (see Section 6 of \cite{jiao2023wright}) and give some basic properties such as exchangeability, limiting distribution of the ancestral line, and comparisons of $E[T_{MRCA}]$.

For $K\in\mathbb{N}$, let $\mathcal{P}_{K}$ be the set of partitions of $\{1,\cdots,K\}$. Define the space of marked partitions as 
\begin{equation}
\mathcal{P}_{K}^{f}=\{(P_{K}, f)|P_{K}\in \mathcal{P}_{K}, f\in [0,\infty)^{|P_{K}|}\},	
\end{equation}
where $f$ is the \textit{flag}, $0$ represents \textit{active}, $\lambda\in (0,\infty)$ represents \textit{dormant} with transition rate $\lambda$, and $|\cdot|$ denotes the number of blocks. For example, when $K=6$, an element $\pi=(P_{6},f)$ may be $\{\{1\}^{0}\{2\}^{0.1}\{5\}^{1}\{3,4\}^{10}\{6\}^{0}\}$. The state space $\mathcal{P}^{f}_{K}$ is locally compact and Polish since it can be regarded as a closed subset of $\{1,2,\cdots,B_{K}\}\times[0,\infty)^{K}$, where $B_{K}$ is the $K$-th Bell number, i.e., the number of different ways to partition a set with $K$-elements.
\begin{defn}\label{defn:coalescent}
For the prescribed finite measure $\mu$ on $((0,\infty),\mathcal{B}(0,\infty))$, where $\mathcal{B}(0,\infty)$ is the Borel $\sigma$-algebra, the continuum seed-bank $K$-coalescent process $\{\Pi^{K}_{t}\}_{t\geq 0}$ is a $\mathcal{P}_{K}^{f}$-valued Markov jump process with initial value $\Pi^{K}_{0}$ composed of $K$ singletons:
\begin{equation}\label{eq:K}
\pi \mapsto \pi^{\prime} \text { at rate }\left\{\begin{array}{cl}
\mu(B), & \pi \leadsto \pi^{\prime}, \text{a}~0~\text{becomes}~\lambda\in B,\text{ for } B\in\mathcal{B}(0,\infty),\\
\lambda, & \pi \leadsto \pi^{\prime}, \text{a}~\lambda ~\text{becomes}~0,\\
1, & \pi \sqsupset \pi^{\prime},
\end{array}\right.
\end{equation}
where $\pi\leadsto\pi^{\prime}$ denotes that $\pi^{\prime}$ is obtained by changing the flag of one block of $\pi$, and $\pi\sqsupset \pi^{\prime}$ denotes that  $\pi^{\prime}$ is obtained by merging two $0$-blocks in $\pi$.
For example, when $K=6$, an event $\pi\sqsupset \pi^{\prime}$ may be
\[\{\{1\}^{0}\{2\}^{0.1}\{5\}^{1}\{3,4\}^{10}\{6\}^{0}\}\sqsupset \{\{1,6\}^{0}\{2\}^{0.1}\{5\}^{1}\{3,4\}^{10}\},\]
and an event $\pi\leadsto\pi^{\prime}$ may be
\[\{\{1\}^{0}\{2\}^{0.1}\{5\}^{1}\{3,4\}^{10}\{6\}^{0}\}\leadsto 
\{\{1\}^{0}\{2\}^{0.1}\{5\}^{1}\{3,4\}^{10}\{6\}^{100}\}.
\]

\end{defn}
If the measure $\mu$ has a continuous part, i.e., non-discrete, then the process $\{\Pi^{K}_{t}\}_{t\geq 0}$ is \textit{continuous in state}, which is the origin of the term \textit{continuum}.

\begin{defn}\label{defn:block}
The block counting process (counting the number of active/dormant blocks in partitions) of the continuum seed-bank $K$-coalescent is the following Markov jump process $\{(N_{t}, M_{t})\}_{t\geq 0}$ with initial value $(N_{0}, M_{0})$ such that 
$N_{0}+||M_{0}||_{TV}=K$, 
\begin{equation}\label{eq:char1}
(n, m) \mapsto (n^{\prime}, m^{\prime}) \text { at rate }\left\{\begin{array}{cl}
n\mu(B), & (n^{\prime}, m^{\prime})=(n-1, m+\delta_{\lambda}), \lambda\in B, \text{for}~B\in\mathcal{B}(0,\infty),\\
\lambda m(\{\lambda\}), & (n^{\prime}, m^{\prime})=(n+1, m-\delta_{\lambda}),\\
C^{2}_{n}, & (n^{\prime}, m^{\prime})=(n-1, m),
\end{array}\right.
\end{equation}
where $n\in \mathbb{N}_{0}:=\{0,1,\cdots\}$, $m$ is a finite integer-valued measure on $(0,\infty)$, and $C^{2}_{n}=\frac{n(n-1)}{2}$. Since $m$ can be represented as a weighted sum of Dirac measures (see e.g. Theorem 2.1.6.2 in \cite{kadets2018course}), i.e., $m(d\lambda)=\sum\limits_{i=1}^{K_{0}}m_{i}\delta_{\lambda_{i}}(d\lambda)$ for $K_{0}\in \mathbb{N}_{0}$, $m_{i}\in\mathbb{N}:=\{1,2,\cdots\}$ and $\lambda_{i}\in (0,\infty)$, the state space for $(n,m)$ will be denoted by $\mathbb{N}_{0}\times \bigoplus\limits_{(0,\infty)}\mathbb{N}_{0}$, where the direct sum $\bigoplus\limits_{(0,\infty)}\mathbb{N}_{0}$ is the subspace of $\mathbb{N}_{0}^{(0,\infty)}$ for whose elements there are finitely many non-zero components.	
\end{defn}
Let $\{\tilde{M}_{k},k\in \mathbb{N}_{0}\}$ be a discrete-time time-homogeneous Markov chain in $\mathbb{N}_{0}\times \bigoplus\limits_{(0,\infty)}\mathbb{N}_{0}$ with initial distribution $\nu$, and its transition kernel is 
\begin{equation}
\mu((n, m), \Gamma)=\left\{\begin{array}{l}
\frac{n \mu(B)}{c_{n,m}}, \Gamma=\left\{\left(n-1, m+\delta_\lambda\right), \lambda \in B\right\}, \\
\frac{\lambda m(\{\lambda \})}{c_{n,m}}, \Gamma=\left(n+1, m-\delta_\lambda\right), \\
\frac{C_n^2}{c_{n,m}}, \Gamma=(n-1, m),
\end{array}\right.
\end{equation}
where $c_{n,m}=cn+\int_{(0,\infty)} \lambda m(d\lambda)+C^2_n$. Then, $\{(N_{t}, M_{t})\}_{t\geq 0}$ can be represented as 
\begin{eqnarray}\label{eq:constructed}
(N_{t}, M_{t})=\sum_{k=0}^{\infty}\tilde{M}_{k}I_{\{[\sum\limits_{j=0}^{k-1}\frac{\Delta_{j}}{c_{Y_{j}}}, \sum\limits_{j=0}^{k}\frac{\Delta_{j}}{c_{Y_{j}}})\}},	
\end{eqnarray}
where $c_{Y_{j}}=c_{n,m}$ if $Y_{j}=(n,m)$, $\{\Delta_{j}\}_{j\in\mathbb{N}_{0}}$ are mutually independent and independent of $\{Y_{j}\}_{j\in\mathbb{N}_{0}}$, and they are exponentially distributed with parameter $1$. The \textit{continuum seed-bank coalescent} can be constructed in a similar way, but it is not needed in this paper.

Finally, since the state space $\mathcal{P}_{K}^{f}$ is Polish, by a projective limit argument, the existence and uniqueness in distribution of the continuum seed-bank coalescent process $\{\Pi^{\infty}_{t}\}_{t\geq 0}$ taking values in $\mathcal{P}^{f}_{\infty}$ is given by the Proposition 6.5. in \cite{jiao2023wright}, which allows us to consider the \textit{coming down from infinity} property when $N_{0}+||M_{0}||_{TV}$ is countably infinite.

\subsection{Duality}

The dual process and a scaling limit interpretation of the \textit{continuum seed-bank coalescent} have been provided by Theorem 1.3 and Proposition 6.6 in \cite{jiao2023wright}. Here, we briefly review the duality relation, as the limiting distribution of the dual process, which will be refereed to as the 
\textit{continuum seed-bank diffusion}, can be derived as the byproduct of Theorem \ref{th:mean} in the present paper, see Proposition \ref{prop:fixation} on \textit{fixation in law}. 

\begin{defn}
Let $D:=[0,1]\times \left\{y:(0, \infty) \rightarrow \mathbb{R} \text { is Borel measurable, and } 0 \leq y \leq 1, \mu \text {-a.e.}\right\}$ equipped with the metric on $\mathbb{R} \times L^{1}((0,\infty),\mathcal{B}(0,\infty),\mu; \mathbb{R})$, and we assume that $\int_{(0,\infty)}\lambda \mu(d\lambda)<\infty$, the \textit{continuum seed-bank diffusion} $\{X_{t}, Y_{t}\}_{t\geq 0}$ is the unique strong solution to the following infinite-dimensional stochastic differential equation (SDE):
\begin{equation}\label{eq:inf}
\left\{\begin{array}{l}
dX_{t}=\int_{(0,\infty)} Y_{t}(\lambda) \mu(d \lambda) d t-c X_{t} d t+\sqrt{X_{t}\left(1-X_{t}\right)} d W_{t}, t>0,\\
dY_{t}(\lambda)=\lambda\left(X_{t}-Y_{t}(\lambda)\right) dt, \lambda\in(0,\infty), t>0, \\
\left(X_{0}, Y_{0}(\lambda)\right)=(x, y(\lambda))\in D.
\end{array}\right.
\end{equation}	
\end{defn}

For the well-posedness of (\ref{eq:inf}), see Theorem 1.1 in \cite{jiao2023wright}. If the initial value does not depend on $\mu$, then the \textit{continuum seed-bank diffusion} takes values in 
\begin{equation}
\overline{D}:=[0,1]\times \left\{y:(0, \infty) \rightarrow \mathbb{R} \text { is Borel measurable, and } 0 \leq y\leq 1\right\},	
\end{equation}
which is equipped with the metric on $\mathbb{R} \times L^{\infty}((0,\infty),\mathcal{B}(0,\infty);\mathbb{R})$. It is well-known that the dual space of $L^{\infty}$ can be viewed as the space of finitely additive set functions on $((0,\infty),\mathcal{B}(0,\infty))$ equipped with the total variation norm $||\cdot||_{TV}$. In this case, according to Remark 6.4 in \cite{jiao2023wright}, if we adopt the $\sigma$-algebra generated by all integer-valued finite measures $m$ on $L^{\infty}$, then the process $Y$ is an $L^{\infty}$-valued adapted process. 

By Theorem 1.3 in \cite{jiao2023wright}, the \textit{continuum seed-bank diffusion} and the \textit{continuum seed-bank coalescent} are dual to each other through the following relation:
\begin{equation}
E[F((X_t, Y_t),(N_{0}, M_{0}))]=E[F((X_{0}, Y_{0}),(N_{t}, M_{t}))], t\geq 0,
\end{equation}
for the dual functional 
\begin{equation}\label{dualfunction}
F[(x, y(\lambda)),(n,m(d\lambda))]:=x^{n}e^{\int_{(0,\infty)} ln y(\lambda)m(d\lambda)}.
\end{equation}
In particular, if $(X_{0}, Y_{0})=(x, y)$ and $(N_{0}, M_{0})=(n, \sum\limits_{i=1}^{K_{0}}m_{i}\delta_{\lambda_{i}})$, then we have
\begin{equation}\label{determining}
E[X_t^n\prod_{i=1}^{K_{0}}Y_{t}(\lambda_{i})^{m_{i}}]=E[x^{N_{t}}\prod_{i=1}^{K_{t}}y(\Lambda_{i,t})^{M_{i,t}}], t\geq 0,
\end{equation}
for $M_{t}=\sum\limits_{i=1}^{K_{t}}M_{i,t}\delta_{\Lambda_{i,t}}.$
By Remark 6.4 in \cite{jiao2023wright}, the distribution of $(X_{t}, Y_{t})$ on $\mathbb{R}\times L^{\infty}$ is uniquely determined by (\ref{determining}).

\subsection{Exchangeability}\label{sub:ex}

We now give a graphical construction of the \textit{continuum seed-bank coalescent}, 
imitating the construction for the Kingman coalescent (see e.g. \cite{berestycki2009recent}).

For a sample of size $K$ listed by $\{1, 2, \cdots, K\}$, at the beginning we have $K$ independent ancestral lines for these individuals since the initial partition $\Pi^{K}_{0}$ is composed of $K$ singletons. Here $K$ can be countably infinite. Each ancestral line follows an \textit{alternating renewal process} (see e.g. \cite{cox1962renewal}), i.e., the individual switches back and forth between \textit{active} and \textit{dormant} states. According to the Definition \ref{defn:coalescent}, the length of active period is exponentially distributed with parameter $c:=\mu(0,\infty)$, the length of dormant period (dormancy time) follows a distribution with CDF $K(t)=\int_{(0,\infty)}(1-e^{-\lambda t})\frac{\mu(d\lambda)}{c}$, and them are all independent. Note that once two blocks coalesce, the corresponding individuals will follow the same ancestral line in the future, but ancestral lines of different blocks are always independent.

For $t\geq 0$, we label every block of the partition $\Pi^{K}_{t}$ by the \textit{position} and \textit{flag} of its lowest elements in the list. In other words, for every $1\leq i\leq K$, we define a label process $\{X_{t}(i)\}_{t\geq 0}$ which takes values in $\{1, 2, \cdots, K\}\times [0, \infty)$ such that $X_{t}(i)=(j, f_{t}(j))$ for some $j\leq i$ implies at time $t$, the lowest element of the block containing $i$ is the $j$-th individual and the flag is $f_{t}(j)$ indicating the state of $j$. Then, at time $0$, we have $X_{0}(i)=(i, f_{i})$ for the initial flags $\mathbf{f}=\{f_{i}\}_{1\leq i\leq K}$. Clearly, as the coalescent process evolves, $X_{t}(i)=(j, f_{t}(j))$ only jumps to a lower position, say $l<j$, when $f_{t}(j)=0$, and at the same time the flag of $X_{t}(l)$ should also be $0$.

Now, we assign independent Possion clocks $\{P_{i,j}\}_{1\leq i<j\leq K}$ with intensity $1$ to each pair of these $K$ individuals. Due to the memoryless property of the exponential distribution, when the Possion clock
$P_{i,j}$ rings, and at the same time $t_{0}$, the flags of both $X_{t_{0}}(i)$ and $X_{t_{0}}(j)$ are $0$ but their positions are different, then we know that a new coalescence is occurring. Given $2\leq n \leq K$, in order to define $\{X_{t}(i)\}_{1\leq i\leq n}$, we only need to consider finitely many Possion clocks $\{P_{i,j}\}_{1\leq i<j\leq n}$ since other clocks do not have any effect. Define 
\[T_{1}=\inf\{t>0: P_{i,j} \text{ rings at } t \text{ and } X_{t}(i)=(p_{i}, 0), X_{t}(j)=(p_{j}, 0) \text{ for } p_{i}\neq p_{j}, \text{ for some } 1\leq i<j\leq n\},  
\]
and recursively, for $k\geq 2$, define
\[T_{k}=\inf\{t>T_{k-1}: P_{i,j} \text{ rings at } t \text{ and } X_{t}(i)=(p_{i}, 0), X_{t}(j)=(p_{j}, 0) \text{ for } p_{i}\neq p_{j}, \text{ for some } 1\leq i<j\leq n\}. 
\]
Then $\{T_{k}\}_{k\in\mathbb{N}}$ is the sequence of coalescence times. If $t<T_{1}$, we already know that $X_{t}(i)=X_{0}(i)=(i, f_{t}(i))$. For $k\geq 1$, if $\{X_{t}(i)\}_{1\leq i\leq n}$ is defined for all $t<T_{k}$, and we assume 
that the Possion clock $P_{i_{k}, j_{k}}$ rings at $T_{k}$ (only one clock except on a zero-probability set). Let $I_{k}=\{1\leq i\leq n: X_{T_{k}-}(i)=(j_{k}, 0)\}$ which is the set of individuals whose label changes at time $T_{k}$. By the definition of $T_{k}$, $I_{k}$ is not empty. Define $X_{t}(i)=X_{T_{k}-}(i)$ if $i\notin I_{k}$ for all $t\in [T_{k}, T_{k+1})$, and $X_{t}(i)=(i_{k}, f_{t}(i_{k}))$ if  $i\in I_{k}$ for all $t\in [T_{k}, T_{k+1})$. 

Finally, with all label processes $\{X_{t}(i), t\geq 0\}_{1\leq i\leq K}$ well-defined, the partition $\Pi^{K}_{t}$ is given by
\[i \text{ and } j\text{ are in the same block if and only if } X_{t}(i)=X_{t}(j).\]

By the above construction, we can easily see that the \textit{continuum seed-bank coalescent} is exchangeable, i.e., for any $t\geq 0$, $\Pi^{K}_{t}$ is an exchangeable random marked partition which is defined as follows:

\begin{defn}
An exchangeable random marked partition $\Pi$ is a random marked partition of $\{1,2,\cdots, K\}$, $K$ can be countably infinite, whose distribution is invariant under any permutation $\sigma$ of 
$\mathbb{N}$ with finite support which only permutes finitely many positions. In the resulting partition $\Pi_{\sigma}$, $\sigma(i)$ and $\sigma(j)$ are in the same block if and only if $i$ and $j$ are in the same block of $\Pi$. 
\end{defn}

\begin{prop}
The \textit{continuum seed-bank coalescent} $\{\Pi^{\infty}_{t}\}_{t\geq 0}$ is exchangeable.
\end{prop}
\begin{proof}
We only need to consider swaps. For example, if individuals $1$ and $2$ are swapped, then they flags indicating the state are swapped at the same time since the flags are not permuted. If we perform this swap at time $0$, then the overall distribution of the label processes $\{X_{t}(i), t\geq 0\}_{1\leq i\leq K}$ are not changed since individuals are homogeneous with the same initial flag, and the Possion clocks are also homogeneous. Therefore, the distribution of $\Pi^{K}_{t}$ for any $t\geq 0$ does not change.
\end{proof}

\subsection{Limiting distribution of the ancestral line}

In order to prove Theorem \ref{th:mean}, we study the limiting distribution of the ancestral line. The ancestral line of individual $i$ refers to the $flag$ component of the label process $\{X_{t}(i)\}_{t\geq 0}$ defined in Subsection 
\ref{sub:ex}, whose distribution is not affected by the coalescent since all individuals involved in a coalescence event are active at that time. As mentioned in Subsection \ref{sub:ex}, the ancestral line of an individual follows an \textit{alternating renewal process} $\{R_{t}\}_{t\geq 0}$ determined by the initial flag. If the initial flag is $0$, i.e., active, then by the classical renewal theory (see e.g. \cite{cox1962renewal}), we have
\[\lim_{t\rightarrow \infty}P(R^{\{0\}}_{t} \text{ is active})=\frac{\frac{1}{c}}{\frac{1}{c}+\int_{(0,\infty)}\frac{1}{\lambda}\frac{\mu(d\lambda)}{c}}=\frac{1}{1+\int_{(0,\infty)}\frac{1}{\lambda}\mu(d\lambda)},\]
where $c$ is the total mass of $\mu$; If the initial flag is $\lambda$, i.e., dormant with transition rate $\lambda$, then 
\[\lim_{t\rightarrow \infty}P(R^{\{\lambda\}}_{t} \text{ is active})=\lim_{t\rightarrow \infty}\int_{0}^{t}P(R^{\{0\}}_{t-s} \text{ is active})\lambda e^{-\lambda s}ds=\frac{1}{1+\int_{(0,\infty)}\frac{1}{\lambda}\mu(d\lambda)}.\]

Moreover, under the condition that $\int_{(0,\infty)}\frac{1}{\lambda}\mu(d\lambda)<\infty$, we have the following result:

\begin{prop}\label{prop:limit}
The limiting distribution of the ancestral line is $\frac{1}{1+\int_{(0,\infty)}\frac{1}{\lambda}\mu(d\lambda)}\delta_{0}+\frac{1}{1+\int_{(0,\infty)}\frac{1}{\lambda}\mu(d\lambda)}\frac{1}{\lambda}.\mu(d\lambda)$, where $\delta_{0}$ is the Dirac measure at $0$, and $\frac{1}{\lambda}.\mu(d\lambda)(B)=\int_{B}\frac{1}{\lambda}\mu(d\lambda)$ for any $B\in\mathcal{B}(0,\infty)$.
\end{prop}
\begin{proof}
Let $(p_{t}, q(t, d\lambda))$ be the absolute distribution of the ancestral line at time $t$. By the Kolmogorov forward equation, it is well-known that $(p_{t}, q(t, d\lambda))$ satisfies the following ordinary differential equation:
\begin{equation*}
\left\{\begin{array}{l}
\frac{dp(t)}{dt}=-cp(t)+\int_{(0,\infty)}\lambda q(t, d\lambda), t>0,\\
\frac{q(t, B)}{dt}=-\int_{B}\lambda q(t, d\lambda)+p(t)\mu(B), B\in\mathcal{B}(0,\infty), t>0,\\
p_{t}+||q(t, d\lambda)||_{TV}=1.
\end{array}\right.
\end{equation*}	
By the variation of constants, in the form of measures, we have
\[q(t, d\lambda)=e^{-\lambda t}q(0,d\lambda)+\int_{0}^{t}e^{-\lambda s}p(t-s)ds.\mu(d\lambda).\]
Letting $t\rightarrow \infty$, by the dominated convergence theorem, then we have
$\lim\limits_{t\rightarrow \infty}q(t, d\lambda)=\frac{1}{1+\int_{(0,\infty)}\frac{1}{\lambda}\mu(d\lambda)}\frac{1}{\lambda}.\mu(d\lambda)$. Here the convergence is for any $B\in\mathcal{B}(0,\infty)$, which is stronger than the weak convergence.
\end{proof}

The following Lemma about the geometric random variable is intuitive:
\begin{lemma}\label{lem:geo}
For a sequence of independent Bernoulli trails, if the success probability of the $n$-th trial $p_{n}\geq p>0$ ($0<p_{n}\leq p$), then the expected time to the first success is less (more) than that of a geometric random variable with parameter $p$.
\end{lemma}
\begin{proof}
The expected time $E[\tau]=p_{1}+\sum\limits_{n=2}^{\infty}np_{n}q_{1}\cdots q_{n-1}$ for $q_{n}=1-p_{n}$. We first decrease $p_{1}$ to $p$ and keep others unchanged, then $E[\tau]$ increases by $p_{1}-p$ but decreases by at least $2(p_{1}-p)$, thus $E[\tau]$ is decreased. Change the success probability one by one, by induction we have the desired result.
\end{proof}

We are now ready to prove Theorem \ref{th:mean}.
\begin{proof}[\textbf{Proof of Theorem \ref{th:mean}}]
If the expected dormancy time $\int_{(0,\infty)}\frac{1}{\lambda}\mu(d\lambda)=\infty$, then once an active individual enters dormancy, the expected time to revive is infinite. Such an event occurs with positive probability, thus we know that $E[T_{MRCA}(n, m)]=\infty$.

If $\int_{(0,\infty)}\frac{1}{\lambda}\mu(d\lambda)<\infty$, then by Proposition
\ref{prop:limit}, after a long time, the absolute distribution of the ancestral line is approximately $\frac{1}{1+\int_{(0,\infty)}\frac{1}{\lambda}\mu(d\lambda)}\delta_{0}+\frac{1}{1+\int_{(0,\infty)}\frac{1}{\lambda}\mu(\lambda)}\frac{1}{\lambda}.\mu(d\lambda)$. If we consider the initial $K=n+||m||_{TV}$ individual by pairs, then we have
\[T_{MRCA}(n,m)=\max_{1\leq i<j\leq K}T_{MRCA}^{i, j}\leq \sum_{1\leq i<j\leq K}T_{MRCA}^{i, j},\]
where $T_{MRCA}^{i, j}$ denotes the time it takes for individuals $i$ and $j$ to coalesce. Given $\epsilon>0$ small enough, for each pair $i$ and $j$, there exits $T_{i,j}(\epsilon)>0$ such that 
\[P(i \text{ is active at } t)\cdot P(j \text{ is active at } t)\geq p_{\epsilon}=(\frac{1}{1+\int_{(0,\infty)}\frac{1}{\lambda}\mu(d\lambda)}-\epsilon)^{2}, \textit{ for } t\geq T_{i,j}(\epsilon).\]
From the construction in Subsection \ref{sub:ex}, we know that if $i$ and $j$ have not coalesced in $[0, T_{i, j}(\epsilon)]$, then starting from $T_{i, j}(\epsilon)$, once the Possion clock $P_{i, j}$ with intensity $1$ rings and they are both active at that time , then the coalescence occurs. Now, the success probability at each time that $P_{i,j}$ rings is bounded blew by $p_{\epsilon}$, then by Lemma \ref{lem:geo}, we will need less average number of attempts than the case when the success probability is exactly $p_{\epsilon}$. Since the time interval $\{\tau_{k}\}_{k\in\mathbb{N}}$ of $P_{i, j}$ are independent exponential random variables with parameter $1$, by
the Wald's identity (see e.g. \cite{cox1962renewal}), the expected time to coalesce is 
\[E[\tau_{1}+\cdots+\tau_{A}]=E[\tau_{1}]E[A]\leq \frac{1}{p_{\epsilon}},\]
where $A$ denotes the number of attempts.

In conclusion, we have
\[E[T_{MRCA}(n,m)]\leq \sum_{1\leq i<j\leq K}E[T_{MRCA}^{i, j}]\leq \sum_{1\leq i<j\leq K}(T_{i,j}(\epsilon)+\frac{1}{p_{\epsilon}})<\infty.\]
The proof is completed.
\end{proof}

\begin{rem}\label{rem:explicit}
For the single seed-bank case, i.e., $\mu=\delta_{\lambda}$, the $E[T_{MRCA}]$ of $2$ individuals can be explicitly calculated by the following recurrence equation:
\begin{equation}
\left\{\begin{array}{l}
E[T_{MRCA}(0,2\delta_{\lambda})]=\frac{1}{2\lambda}+E[T_{MRCA}(1,\delta_{\lambda})],\\
E[T_{MRCA}(1,\delta_{\lambda})]=\frac{1}{c+\lambda}+\frac{c}{c+\lambda}E[T_{MRCA}(0,2\delta_{\lambda})]+\frac{\lambda}{c+\lambda}E[T_{MRCA}(2,0)],\\
E[T_{MRCA}(2,0)]=\frac{1}{1+2c}+\frac{2c}{1+2c}E[T_{MRCA}(1,\delta_{\lambda})].
\end{array}\right.
\end{equation}	
The solution is $E[T_{MRCA}(0,2\delta_{\lambda})]=1+\frac{4c+3}{2\lambda}+\frac{2c^{2}+c}{2\lambda^{2}}$. 	
\end{rem}

As a byproduct of Theorem \ref{th:mean}, the limiting distribution of the \textit{continuum seed-bank diffusion} (see Definition \ref{eq:inf}) can be derived. See Corollary 2.10 in \cite{blath2016} for the single seed-bank case.

\begin{prop}\label{prop:fixation}[\textit{fixation in law}]
If $\int_{(0,\infty)}\frac{1}{\lambda}\mu(d\lambda)<\infty$ and $(X_{0}, Y_{0})=(x, y)\in \overline{D}$, then as $t\rightarrow\infty$, $(X_{t},Y_{t})$ converges in distribution to a $\mathbb{R}\times L^{\infty}$-valued random variable $(X_{\infty}, Y_{\infty})$ whose distribution is given by
	\[\frac{x+\int_{(0,\infty)}y(\lambda)\frac{1}{\lambda}\mu(d\lambda)}{1+\int_{(0,\infty)}\frac{1}{\lambda}\mu(d\lambda)}\delta_{(1,1)}+\frac{1-x+\int_{(0,\infty)}(1-y(\lambda))\frac{1}{\lambda}\mu(d\lambda)}{1+\int_{(0,\infty)}\frac{1}{\lambda}\mu(d\lambda)}\delta_{(0,0)}.\]	 
\end{prop} 
\begin{proof}
By Remark 6.4 in \cite{jiao2023wright}, the distribution of $(X_{t}, Y_{t})$ on $\mathbb{R}\times L^{\infty}$ is uniquely determined by (\ref{determining}). For $m=\sum\limits_{i=1}^{K_{0}}m_{i}\delta(\lambda_{i})$ and $T=T_{MRCA}(n,m)$, by $E[T]<\infty$ and $(x,y)\in D$, we have
\begin{eqnarray*}
&&\lim_{t\rightarrow\infty}E[X_t^n\prod_{i=1}^{K_{0}}Y_{t}(\lambda_{i})^{m_{i}}]=\lim_{t\rightarrow\infty}E[x^{N_{t}}\prod_{i=1}^{K_{t}}y(\Lambda_{i,t})^{M_{i,t}}]\\
&=&\lim_{t\rightarrow\infty}E[x^{N_{t}}\prod_{i=1}^{K_{t}}y(\Lambda_{i,t})^{M_{i,t}}|T\leq t]P(T\leq t)+\lim_{t\rightarrow\infty}E[x^{N_{t}}\prod_{i=1}^{K_{t}}y(\Lambda_{i,t})^{M_{i,t}}|T>t]P(T> t)\\
&=& \lim_{t\rightarrow\infty}(xP(N_{t}=1, T\leq t)+\int_{(0,\infty)}y(\lambda)P(M_{t}=1, T\leq t, d\lambda))\\
&=& \frac{x+\int_{(0,\infty)}y(\lambda)\frac{1}{\lambda}\mu(d\lambda)}{1+\int_{(0,\infty)}\frac{1}{\lambda}\mu(d\lambda)},
\end{eqnarray*}
where the measure $P(M_{t}=1, T\leq t, d\lambda)$ corresponds to $q(t, d\lambda)$
in Proposition \ref{prop:limit}. Therefore, the limiting moments of $(X_{t}, Y_{t})$ are independent of $(n, m)$ and always equal $\frac{x+\int_{(0,\infty)}y(\lambda)\frac{1}{\lambda}\mu(d\lambda)}{1+\int_{(0,\infty)}\frac{1}{\lambda}\mu(d\lambda)}$, which uniquely characterizes the distribution given in the theorem.
\end{proof}

\subsection{Comparisons of $E[T_{MRCA}]$}

In this subsection, we discuss the problem of comparing the $E[T_{MRCA}]$ of different cases. Although some seemingly intuitive statements are still open problems, they have no applications in the subsequent context. For the proof of Theorem \ref{th:bounds}, we only need the following elementary lemma:

\begin{lemma}\label{lem:compare}~
\begin{enumerate}
	\item For fixed measure $\mu$, $E[T_{MRCA}(2,0)]\leq E[T_{MRCA}(1,\delta_{\lambda})] \leq E[T_{MRCA}(0,2\delta_{\lambda})]$ for any $\lambda\in (0,\infty)$;
	\item If the support of $\mu$ has a lower bound $\underline{\lambda}_{0}>0$, then we have
	\begin{equation}\label{eq:1}
		E[T_{MRCA}^{\{\mu\}}(1,\delta_{\lambda})]\leq E[T_{MRCA}^{\{\mu\}}(1,\delta_{\underline{\lambda}_{0}})]\leq
	E[T_{MRCA}^{\{c\delta_{\underline{\lambda}_{0}}\}}(1,\delta_{\underline{\lambda}_{0}})],
	\end{equation}
	\begin{equation}\label{eq:2}
	E[T_{MRCA}^{\{\mu\}}(0,\delta_{\lambda}+\delta_{\lambda^{\prime}})]\leq  E[T_{MRCA}^{\{\mu\}}(0, 2\delta_{\underline{\lambda}_{0}})]\leq E[T_{MRCA}^{\{c\delta_{\underline{\lambda}_{0}}\}}(0,2\delta_{\underline{\lambda}_{0}})],
	\end{equation}
	where the superscript indicates the measure describing the transition, $c\delta_{\underline{\lambda}_{0}}$ implies the singe seed-bank case,  
	and $\lambda\geq \underline{\lambda}_{0}$. 
\end{enumerate}
\end{lemma}
\begin{proof}~

1. This is obvious since for two individuals, it is necessary for the block counting process to pass through the state $(2,0)$ for the coalescence. If we start from $(0,2\delta_{0})$, then the next state must be $(1,\delta_{0})$, thus the former requires more time in average.

2. Note that the coalescent process can be categorized into $3$ levels based on the number of active individuals, i.e., 
	\[(0,\delta_{\lambda}+\delta_{\lambda^{\prime}})\rightleftharpoons (1,\delta_{\lambda^{\prime\prime}})\rightleftharpoons(2,0),\] 
	Once the process attains the level $(2,0)$, the probability to coalesce is $\frac{1}{1+2c}$ independent of the measure $\mu$. In order to attain $(2,0)$, at the middle level $(1,\delta_{\lambda^{\prime\prime}})$, the success probability probability is $\frac{\lambda^{\prime\prime}}{c+\lambda^{\prime\prime}}$. Even it fails and then goes to the first level $(0,\delta_{\lambda}+\delta_{\lambda^{\prime\prime
	}})$ for some new $\lambda\in (0,\infty)$, it will come back the middle level at the next step to states $(1,\delta_{\lambda})$ or $(1,\delta_{\lambda^{\prime\prime}})$. Since $\underline{\lambda}_{0}$ is 
	the lower bound, the success probability is always greater than $\frac{\underline{\lambda}_{0}}{c+\underline{\lambda}_{0}}$, the expected sojourn time at the middle level is always less than $\frac{1}{c+\underline{\lambda}_{0}}$, and the expected sojourn time at the first level is always less than $\frac{1}{2\underline{\lambda}_{0}}$.
	
	Now, for the first inequality in (\ref{eq:1}), the only difference is the initial value. By Lemma \ref{lem:geo}, $(1,\delta_{\underline{\lambda}_{0}})$ takes the most steps in average to first attain $(2,0)$ (the first \textit{round}) after which there is no difference anymore, and meanwhile the expected sojourn time at each step is the longest. For the second inequality in (\ref{eq:1}), we only need to consider more rounds and the same argument holds for each of them. The first inequality in (\ref{eq:2}) follows by (\ref{eq:1}) and the fact that $(0,\delta_{\lambda}+\delta_{\lambda^{\prime}})$ takes less time in average to the middle level than from $(0,2\delta_{\underline{\lambda}_{0}})$ to $(1,\delta_{\underline{\lambda}_{0}})$. The second inequality in (\ref{eq:2}) directly comes from (\ref{eq:1}). 
\end{proof}

We discuss some open problems in the following Remark.
\begin{rem}\label{rem:open}~
\begin{enumerate}
	\item The monotonicity of $E[T_{MRCA}^{\{\mu\}}(1,\delta_{\lambda})]$ with respect to $\lambda\in (0,\infty)$. 
	
	Since the measure $\mu$ is fixed, we only need to consider the expected time from $(1,\delta_{\lambda})$ to $(2,0)$, which is denoted by $f(\lambda)$.
	If we denote by $f(\lambda,\lambda^{\prime})$ the expected time from $(0,\delta_{\lambda}+\delta_{\lambda^{\prime}})$ to $(2,0)$, then we have the following recurrence equation:
\begin{equation}\label{eq:2eq}
\left\{\begin{array}{l}
f(\lambda)=\frac{1}{c+\lambda}+\frac{1}{c+\lambda}\int_{(0,\infty)}f(\lambda,\lambda^{\prime})\mu(\lambda^{\prime}),\\
f(\lambda,\lambda^{\prime})=\frac{1}{\lambda+\lambda^{\prime}}+\frac{\lambda^{\prime}}{\lambda+\lambda^{\prime}}f(\lambda)+\frac{\lambda}{\lambda+\lambda^{\prime}}f(\lambda^{\prime}).
\end{array}\right.
\end{equation}	
By direct calculations, we have
\begin{equation}\label{eq:recur}
f(\lambda)=\frac{1}{\lambda}+\frac{\int_{(0,\infty)}\frac{1}{\lambda+\lambda^{\prime}}f(\lambda^{\prime})\mu(d\lambda^{\prime})}{1+\int_{(0,\infty)}\frac{1}{\lambda+\lambda^{\prime}}\mu(d\lambda^{\prime})}.
\end{equation}
For the single seed-bank case, the solution is $f(\lambda)=\frac{1}{\lambda}+\frac{c}{2\lambda^{2}}$, thus the statement is true. For finitely many seed-banks, i.e, $\mu=\sum\limits_{i=1}^{n}c_{i}\delta_{\lambda_{i}}$, (\ref{eq:recur}) is a system of linear equations which can also be explicitly solved . However, for general $\mu$, it is not clear although we firmly believe its correctness. Since $f(\lambda)$ is unbounded, the method of approximating the measure seems not practical.

\item The monotonicity of $E[T_{MRCA}^{\{\mu\}}(0,\delta_{\lambda}+\delta_{\lambda^{\prime}})]$ with respect to $\lambda\in (0,\infty)$ for fixed $\lambda^{\prime}\in (0,\infty)$.

Following the notation in 1, we need to show the monotonicity of $f(\lambda, \lambda^{\prime})$ as the function of $\lambda$. By the second equation in 
(\ref{eq:2eq}), assuming that $f(\lambda)$ is differentiable which is not verified yet, we have
\[\frac{\partial f(\lambda,\lambda^{\prime})}{\partial \lambda}=\frac{-1+\lambda^{\prime}(f(\lambda^{\prime})-f(\lambda))+\lambda^{\prime}(\lambda+\lambda^{\prime})\frac{df(\lambda)}{d\lambda}}{(\lambda+\lambda^{\prime})^{2}},\]
which is negative if $\frac{df(\lambda)}{d\lambda}\leq 0$ and $\lambda\leq \lambda^{\prime}$. Therefore, this statement also relies on the study of the properties of $f(\lambda)$.

\item $E[T_{MRCA}^{\{\mu\}}(2,0)]\leq E[T_{MRCA}^{\{\mu^{\prime}\}}(2,0)]$ if
	$\mu^{\prime}$ is first-order dominated by $\mu$.

This conjecture arises from the fact that $\lim\limits_{t\rightarrow \infty}P_{t}=\frac{1}{1+\int_{(0,\infty)}\frac{1}{\lambda}\mu(d\lambda)}$ for $P_{t}:=P(\text{the individual is active at }t)$ (see Proposition \ref{prop:limit}). If 
$\int_{(0,\infty)}\frac{1}{\lambda}\mu^{\prime}(d\lambda)\geq \int_{(0,\infty)}\frac{1}{\lambda}\mu(d\lambda)$, then after a long period of time, the two individuals described by $\mu$ are more likely to coalesce, see the proof of Theorem \ref{th:mean}. From the definition of the first-order stochastic dominance, we have $\mu(\lambda,\infty)\geq \mu^{\prime}(\lambda,\infty)$ for any $\lambda\in(0,\infty)$. Then, through integration by parts, we know the desired condition holds. However, the comparison of $P_{t}$ before being stationary is not clear. By the classical renewal theory (see e.g. \cite{cox1962renewal}), $P_{t}$ is the solution to the renewal equation
\[P_{t}=e^{-ct}+\int_{0}^{t}P_{t-s}dK(s), t\geq 0,\]
where $K(t)=\int_{(0,\infty)}(1-e^{-\lambda t})\frac{\mu(d\lambda)}{c}$ is the CDF of the dormancy time. Taking the Laplace transform on both sides, we have 
\[\mathcal{L}(P)(\xi)=\frac{1}{\int_{(0,\infty)}\frac{\xi(\xi+\eta+c)}{\xi+\eta}\frac{\mu(d\eta)}{c}},\]
Even for the single seed-bank case where $P_{t}=\frac{\lambda}{c+\lambda}+\frac{c}{c+\lambda}e^{-(c+\lambda)t}$ can be obtained by inversion, $P_{t}$ is not increasing with respect to $\lambda$.  
\end{enumerate}

\end{rem}
 
\section{Accelerated coalescent}
In \cite{blath2016}, the proof of \textit{not coming down from infinity} is based on the coupling with a \textit{colored seed-bank coalescent} $\{\underline{\Pi}_{t}\}_{t\geq 0}$ which behaves like the \textit{seed-bank coalescent} except that the 
individuals are marked with a color (\textit{white} or \textit{blue}) to indicate whether they have left the seed-bank at least once. At the beginning, all individuals, whether active or dormant, are colored \textit{white}. During the coalescent process, whenever a dormant block leaves the seed-bank, all individuals within it are colored blue. In other words, if we only consider those white individuals, they will disappear after leaving the seed-bank. We have a similar construction for the more general continuum case.

In the present paper, we refer to the coalescent process involving only white individuals as the \textit{accelerated coalescent} whose block counting process (counting the number of active/dormant blocks in partitions) $\{(\underline{N}_{t}, \underline{M}_{t})\}_{t\geq 0}$ is the Markov jump process described by 
\begin{equation}
(n, m) \mapsto (n^{\prime}, m^{\prime}) \text { at rate }\left\{\begin{array}{cl}
n\mu(B), & (n^{\prime}, m^{\prime})=(n-1, m+\delta_{\lambda}), \lambda\in B, \text{for}~B\in\mathcal{B}(0,\infty),\\
\lambda m(\{\lambda\}), & (n^{\prime}, m^{\prime})=(n, m-\delta_{\lambda}),\\
C^{2}_{n}, & (n^{\prime}, m^{\prime})=(n-1, m),
\end{array}\right.
\end{equation}
The coalescent process itself can then be defined imitating Definition \ref{defn:coalescent}.

The term \textit{accelerated} arises from the fact that it is faster than the \textit{continuum seed-bank coalescent} since individuals will never participate in the coalescence once they enter dormancy. More precisely, we have the following relation
\begin{equation}\label{eq:couple}
P(N_{t}^{(n,m)}\geq \underline{N}_{t}^{(n,m)}\text{ and } ||M_{t}^{(n,m)}||_{TV}\geq ||\underline{M}_{t}^{(n,m)}||_{TV} \text{ for any }t\geq 0)=1,
\end{equation}
where $\{N_{t}, M_{t}\}_{t\geq 0}$ is the block counting process of the \textit{continuum seed-bank coalescent}, and the superscript $(n,m)$ indicates the initial value. By the same argument as for $\{(N_{t}, M_{t})\}_{t\geq 0}$, $n+||m||_{TV}$ can also be countably infinite for $\{(\underline{N}_{t}, \underline{M}_{t})\}_{t\geq 0}$.

\subsection{Not coming down from infinity}

For the more general continuum case, the property can be proved with slight modifications to the proof of Theorem 4.1 in \cite{blath2016}. For the convenience of the readers, we still provide a complete proof. Readers who are familiar with \cite{blath2016} can skip Lemma \ref{lem:A}.

The following lemma demonstrates that there are infinitely many individuals have entered the seed-banks with flags in $(0,\lambda_{0}]$ for some $\lambda_{0}>0$ at any time $t>0$. By the definition of $\{(\underline{N}_{t}, \underline{M}_{t})\}_{t\geq 0}$, $\underline{N}_{t}$ is non-increasing, and when it decreases, there are two possibilities: coalescence or deactivation (entering dormancy).
\begin{lemma}\label{lem:A}
For $n\in \mathbb{N}\cup\{\infty\}$, let $c_{0}:=\mu((0,\lambda_{0}])>0$
\[\tau^{n}_{j}:=\inf\{t\geq 0: \underline{N}_{t}^{(n,0)}=j\} \text{ for } 1\leq j\leq n-1, j<\infty,\] 
\[X^{n}_{j}=I_{\{\text{deactivation at }\tau^{n}_{j-1} \text{ to }(0,\lambda_{0}]\}} 
\text{ for } 2\leq j\leq n, j<\infty,\] 
and define $A^{n}_{t}=\sum\limits_{j=2}^{\infty}X^{n}_{j}I_{\{\tau^{n}_{j-1}\leq t\}}$ which is the number of deactivations to $(0,\lambda_{0}]$ before time $t$, then we have
\begin{equation}
P(A^{\infty}_{t}=\infty, \text{ for any } t\geq 0)=1.	
\end{equation} 
\end{lemma}

\begin{proof}
The stopping time $\tau^{n}_{j}$ is finite a.s. since $\tau^{n}_{j}-\tau^{n}_{j-1}$ is exponentially distributed with parameter $\lambda_{j}:=C^{2}_{j}+cj$, where $c$ is the total mass of $\mu$. At each time $\tau^{n}_{j}$, the probability that a deactivation to $(0,\lambda_{0}]$ happens is $\frac{c_{0}j}{\lambda_{j}}=\frac{2c_{0}}{j+2c-1}$ for $c_{0}=\mu((0,\lambda_{0}])$, thus $\{X^{n}_{j}\}_{2\leq j\leq n}$ are independent Bernoulli random variables with parameters $\frac{2c_{0}}{j+2c-1}$, respectively. Note that $\underline{N}_{t}^{(n,0)}$ decreases faster than the Kingman coalescent since some individuals are vanishing by deactivations. As is well-known, Kingman coalescent $\{\tilde{\Pi}^{n}_{t}\}_{t\geq 0}$ comes down from infinity, and hence we have
\begin{eqnarray}\label{eq:A}
\lim_{n\rightarrow \infty}P(A^{n}_{t}\geq \sum_{j=\lfloor \log n\rfloor}^{n}X^{n}_{j})&\geq& \lim_{n\rightarrow \infty}P(\tau^{n}_{\lfloor \log n-1\rfloor}\leq t)\nonumber\\
&\geq& \lim_{n\rightarrow \infty}P(|\tilde{\Pi}^{n}_{t}|\leq \lfloor \log n-1\rfloor)\geq \lim_{n\rightarrow \infty}P(|\tilde{\Pi}^{\infty}_{t}|\leq \lfloor \log n-1\rfloor)=1,	
\end{eqnarray}
where $\lfloor\cdot\rfloor$ is the floor function. 

For any given $\epsilon\in(0,2c_{0})$, by $E[\sum\limits_{j=\lfloor \log n\rfloor}^{n}X^{n}_{j}]=\sum\limits_{j=\lfloor \log n\rfloor}^{n}\frac{2c_{0}}{j+2c-1}$, and \[2c_{0}[\log(n+2c)-\log(\log n+2c-1)]\leq \sum_{j=\lfloor \log n\rfloor}^{n}\frac{2c_{0}}{j+2c-1}\leq 2c_{0}[\log(n+2c-1)-\log(\log n+2c-3)],\]
we have $E[\sum\limits_{j=\lfloor \log n\rfloor}^{n}X^{n}_{j}]\geq (2c_{0}-\epsilon)\log n$
and $Var[\sum\limits_{j=\lfloor \log n\rfloor}^{n}X^{n}_{j}]\leq E[\sum\limits_{j=\lfloor \log n\rfloor}^{n}X^{n}_{j}]\leq 2c_{0}\log n$ when $n$ is large enough. By Chebyshev's inequality, 
\begin{equation}\label{eq:cheby}
P(\sum_{j=\lfloor \log n\rfloor}^{n}X^{n}_{j}\leq c_{0}\log n)\leq P(|\sum_{j=\lfloor \log n\rfloor}^{n}X^{n}_{j}-E[\sum_{j=\lfloor \log n\rfloor}^{n}X^{n}_{j}]|\geq (c_{0}-\epsilon)\log n)\leq \frac{2c_{0}}{(c_{0}-\epsilon)^{2}\log n},
\end{equation}
thus for any $K>0$, we have $\lim\limits_{n\rightarrow \infty}P(\sum\limits_{j=\lfloor \log n\rfloor}^{n}X^{n}_{j}\leq K)=0$. By (\ref{eq:A}), we finally have 
\[\lim_{n\rightarrow \infty}P(A^{\infty}_{t}\leq K)\leq \lim_{n\rightarrow \infty}P(A^{n}_{t}\leq K)=0,\]
and hence $P(A^{\infty}_{t}=\infty, \text{ for any } t\geq 0)=1$.
\end{proof}

The remaining step is show that there are still infinitely many individuals staying in the seed-banks with flag in $(0,\lambda_{0}]$.

\begin{proof}[\textbf{Proof of Theorem \ref{th:coming}}]
By \ref{eq:couple}, it suffices to prove the theorem for the \textit{accelerated coalescent}, and we assume that $\underline{M}_{0}=0$ since more initial dormant individuals  will only contribute to the desired result. For $t>0$, we define
\[\mathcal{B}_{t}:=\{B\subseteq \mathbb{N}: B=B^{\{r, f\}} \text{ is a dormant  block with flag } f\in (0,\lambda_{0}], \text{ and it entered dormancy at } r\in (0,t]\},\]
i.e., the collection of all dormant blocks before time $t$. Then, for any $K>0$, we have
\begin{eqnarray*}
		P(||\underline{M}^{(\infty, 0)}_{t}||_{TV} \leq K)&\leq& P(\sum_{B\in \mathcal{B}^{n}_{t}}I_{\{B\text{ still exists at } t\}}\leq K)\\
		&=&\sum_{i=1}^{K}\sum_{\{\lambda_{l_1},\cdots \lambda_{l_i}\}}\prod_{j=1}^{i}e^{-\lambda_{l_j}(t-r_{l_j})}\prod_{j=i+1}^{n}(1-e^{-\lambda_{l_j}(t-r_{l_j})})\\
		&\leq& \sum_{i=1}^{K}C^{i}_{n}(1-e^{-\lambda_{0}t})^{n-i}\rightarrow 0 \text{ as }n\rightarrow \infty,
	\end{eqnarray*}

where $\mathcal{B}^{n}_{t}=\{B^{\{r_{1},\lambda_{1}\}}, B^{\{r_{2},\lambda_{2}\}},\cdots B^{\{r_{n},\lambda_{n}\}}, r_{j}\in (0,t] \text{ and }\lambda_{j}\in (0,\lambda_{0}],\text{ for }j=1,2,\cdots,n\}$ is the first $n$ blocks in $\mathcal{B}_{t}$. It comes from the facts that the cardinality of $\mathcal{B}_{t}$ is equal to $A^{\infty}_{t}$ which is countably infinite a.s. by Lemma \ref{lem:A}, and $\{I_{\{B^{\{r_{j},\lambda_{j}\}}\text{ still exists at }t\}}\}_{1\leq j\leq n}$ are independent Bernoulli random variables such that 
\[P(B^{\{r_{j},\lambda_{j}\}} \text{ still exists at }t)=e^{-\lambda_{j}(t-r_{j})},\]
as it has been dormant with flag $\lambda_{j}$ during $[r_{j}, t]$. Since $K$ is arbitrary, we know that $P(||\underline{M}^{(\infty, 0)}_{t}||_{TV}=\infty)=1$ for any $t\geq 0$, and then we have
\[P(\underline{N}^{(\infty,0)}_{t}+||\underline{M}_{t}||^{(\infty,0)}_{TV}=\infty, \text{ for any } t\geq 0)=1\] 
by the monotonicity. The proof is completed.
\end{proof}
	
\subsection{Lower bound of $E[T_{MRCA}]$}

In this subsection, we assume that the transition rates of $||m||_{TV}$ dormant individuals are independent $\nu$-distributed random variables for $\nu=\frac{\mu}{c}$. An asymptotic lower bound of $E[T_{MRCA}(n,m)]$ can be obtained by that of the \textit{accelerated coalescent} which is denoted by $E[\underline{T}_{MRCA}(n,m)]$.
\begin{proof}[\textbf{Proof of Theorem \ref{th:bounds}-1}]
Let $A^{n}:=\sum\limits_{j=2}^{n}I_{\{\text{deactivation at }\tau^{n}_{j-1} \}}$, where $\{\tau^{n}_{j-1}\}_{2\leq j\leq n}$ are defined in Lemma \ref{lem:A}. 
Similarly, we have 
\[2c(\log(n+2c)-\log(1+2c))\leq E[A^{n}]=\sum\limits_{j=2}^{n}\frac{2c}{j+2c-1}\leq 2c(\log(n+2c-1)-\log(2c)).\]
For any given $\epsilon\in (0,2c)$, when $n$ is large enough, we have
\[(2c-\epsilon)\log n\leq E[A^{n}]\text{ and } Var[A^{n}]\leq E[A^{n}]\leq 2c\log n.\]
By Chebyshev's inequality, as in (\ref{eq:cheby}), we then obtain
\begin{eqnarray}
\lim_{n\rightarrow \infty}P(A^{n}&\geq & (2c+\epsilon)\log n)=0,\label{eq:Abound1}\\
\lim_{n\rightarrow \infty}P(A^{n}&\leq & (2c-2\epsilon)\log n)=0.\label{eq:Abound2}
\end{eqnarray}
By (\ref{eq:couple}), we know that $E[T_{MRCA}(n,m)]\geq E[\underline{T}_{MRCA}(n,m)]$. If we rewind all of the $A^{n}+||m||_{TV}$ dormant blocks to time $0$, then a lower bound of $E[\underline{T}_{MRCA}(n,m)]$ is given by the time it takes for them to become active. We already know that the transition rates of $A^{n}$ dormant blocks are independent $\nu$-distributed random variables, and now we assume the same condition for the initial $||m||_{TV}$ individuals. Therefore, $E[\underline{T}_{MRCA}(n,m)]$ is bounded below by the extinction time of the pure death process composed of these independent $A^{n}+||m||_{TV}$ blocks whose lifespans are identically distributed with CDF $K(t)=\int_{(0,\infty)}(1-e^{-\lambda t})\frac{\mu}{c}(d\lambda)$.

When there are $i$ remaining blocks, we denote by $T_{i}$ the occurrence time of the first death, the probability that there is no one \textit{death} by time $t$ is 
\[P(T_{i}\geq t)=\prod\limits_{j=1}^{i}(1-F(t))=(\int_{(0,\infty)}e^{-\lambda t}\frac{\mu(d\lambda)}{c})^{i}.\]    
As the result, we have
\begin{eqnarray*}
E[\underline{T}_{MRCA}(n,m)]&\geq& E[\underline{T}_{MRCA}(n,m)I_{\{A^{n}\geq (2c-2\epsilon)\log n \}}]	\\
&=&E[\underline{T}_{MRCA}(n,m)|A^{n}\geq (2c-2\epsilon)\log n]P(A^{n}\geq (2c-2\epsilon)\log n)\\
&\geq &[\sum_{i=1}^{\lfloor(2c-2\epsilon)\log n\rfloor+||m||_{TV}}\int_{0}^{\infty}(\int_{(0,\infty)}e^{-\lambda t}\frac{\mu(d\lambda)}{c})^{i}dt] P(A^{n}\geq (2c-2\epsilon)\log n).
\end{eqnarray*}

If the support of $\mu$ has an upper bound $\overline{\lambda}_{0}$, then by $\int_{(0,\infty)}e^{-\lambda t}\frac{\mu(d\lambda)}{c}\geq e^{-\overline{\lambda}_{0}t}$ for $t\geq 0$, we have
\begin{eqnarray*}
E[\underline{T}_{MRCA}(n,m)]&\geq &(\sum_{i=1}^{\lfloor(2c-2\epsilon)\log n\rfloor+||m||_{TV}}\frac{1}{i\overline{\lambda}_{0}})P(A^{n}\geq (2c-2\epsilon)\log n)
\\
&\geq &\log((2c-2\epsilon)\log n+||m||_{TV})\frac{1}{\overline{\lambda}_{0}}P(A^{n}\geq (2c-2\epsilon)\log n)\\
&\geq &(\log(\log n+\frac{||m||_{TV}}{2c})+\log(2c-2\epsilon))\frac{1}{\overline{\lambda}_{0}}P(A^{n}\geq (2c-2\epsilon)\log n)
\end{eqnarray*}

If $\mu$ is the Gamma distribution $\Gamma(a, b), a>0, b>0$, then we have
\begin{eqnarray*}
E[\underline{T}_{MRCA}(n,m)]&\geq & [\sum_{i=1}^{\lfloor(2c-2\epsilon)\log n\rfloor+||m||_{TV}}\int_{0}^{\infty}\frac{1}{(1+\frac{t}{b})^{ai}}dt]P(A^{n}\geq (2c-2\epsilon)\log n)\\
&=&(\sum_{i=1}^{\lfloor(2c-2\epsilon)\log n\rfloor+||m||_{TV}}\frac{b}{ai-1})P(A^{n}\geq (2c-2\epsilon)\log n)\\
&\geq & \log ((2c-2\epsilon)\log n+||m||_{TV})\frac{b}{a}P(A^{n}\geq (2c-2\epsilon)\log n)\\
&\geq &\log(\log n+\frac{||m||_{TV}}{2c})+\log(2c-2\epsilon))\frac{b}{a}P(A^{n}\geq (2c-2\epsilon)\log n) 
\end{eqnarray*}

Finally, The desired results follow by (\ref{eq:Abound2}).
\end{proof}

\section{Decelerated coalescent}

To obtain an asymptotic upper bound, we need to suitably decelerate the \textit{continuum seed-bank coalescent}. Following the method of \cite{blath2016} again, we define the following \textit{decelerated coalescent} which has restrictions for coalescence. Same as the \textit{accelerated coalescent}, we only give the description of the block counting process. The coalescent process itself can then be defined imitating Definition \ref{defn:coalescent}. 

The block counting process (counting the number of active/dormant blocks in partitions) $\{(\overline{N}_{t}, \overline{M}_{t})\}_{t\geq 0}$ is the Markov jump process described by
\begin{equation}
(n, m) \mapsto (n^{\prime}, m^{\prime}) \text { at rate }\left\{\begin{array}{cl}
n\mu(B), & (n^{\prime}, m^{\prime})=(n-1, m+\delta_{\lambda}), \lambda\in B, \text{for} B\in\mathcal{B}(0,\infty),\\
\lambda m(\{\lambda\}), & (n^{\prime}, m^{\prime})=(n+1, m-\delta_{\lambda}),\\
C^{2}_{n}I_{\{n\geq \lceil(n+||m||_{TV})^{\alpha}\rceil\}}, & (n^{\prime}, m^{\prime})=(n-1, m), \text{ if } n+||m||_{TV}\geq m_{0},\\
C^{2}_{n}, & (n^{\prime}, m^{\prime})=(n-1, m), \text{ if }n+||m||_{TV}< m_{0},
\end{array}\right.
\end{equation}
where $\lceil\cdot\rceil$ is the ceiling function, $\alpha\in (\frac{1}{2}, 1)$, and $m_{0}\in \mathbb{N}$ is a given threshold. 

That is, if the total number of blocks is not less than the threshold $m_{0}$, then a coalescence occurs only when $n\geq \lceil(n+||m||_{TV})^{\alpha}\rceil$.
There are two points here worth further explanations. 
\begin{itemize}
	\item One point is the restricted coalescence rate $C^{2}_{n}I_{\{n\geq \lceil(n+||m||_{TV})^{\alpha}\rceil\}}$ which is the key idea proposed by \cite{blath2016}. In that paper, they take $\alpha=\frac{1}{2}$ for the reason that "the coalescent now happens at a rate that is of the same order of the rate of migration from seed to plant". From an intuitive perspective, the higher the order of this restriction, the more migrations are needed to fulfill the requirement; the lower the order of 
this restriction, the lower the probability of coalescence at the critical state when $n=\lceil(n+||m||_{TV})^{\alpha}\rceil$. From a proof perspective, we need to estimate the time required to meet the restriction as well as the 
time elapsed until the next coalescence occurs after that. Actually, the latter is more crucial, which is the reason for taking $\alpha>\frac{1}{2}$ in the present paper. 
\item Another point is the threshold $m_{0}$, In \cite{blath2016}, there is no such setting since the proof is carried out for the special case when $c=\lambda=1$, and the general case is left for readers. If $c$ and $\lambda$ are not equal, especially when $\lambda<c$, the proof only works for large enough $||m||_{TV}$. Therefore, we introduce a threshold to distinguish those finite exceptional cases which are negligible as $||m||_{TV}$ goes to infinity.
\end{itemize}

The following lemma is Lemma 4.10 in \cite{blath2016} with a slight modification to incorporate the threshold $m_{0}$, by which we know that
\[E[T_{MRCA}(n,m)]\leq E[\overline{T}_{MRCA}(n,m)].\]
\begin{lemma}\label{lem:upper}
\begin{equation}\label{eq:upper}
P(N_{t}^{(n,m)}\leq \overline{N}_{t}^{(n,m)}\text{ and } ||M_{t}^{(n,m)}||_{TV}\leq ||\overline{M}_{t}^{(n,m)}||_{TV} \text{ for any }t\geq 0)=1,  
\end{equation}
where $\{N_{t}, M_{t}\}_{t\geq 0}$ is the block counting process of the \textit{continuum seed-bank coalescent}, and the superscript $(n,m)$ indicates the initial value.
\end{lemma}
\begin{proof}
The proof is basically the same as that of Lemma 4.10 except that 
when the total number of blocks attains $m_{0}-1$, there will be no restriction on the coalescence any more. In the language of Lemma 4.10, before the state $m_{0}-1$, the \textit{continuum seed-bank coalescent} corresponds to the process only considering blue blocks. Hence, at the state $m_{0}-1$, there may exist some additional white blocks. Since there are no differences afterwards, we will always have those white blocks until they coalesce with the blue ones, thus the number of active/dormant blocks for the \textit{decelerated coalescent} is not less than that of the \textit{continuum seed-bank coalescent}.
\end{proof}

\subsection{Upper bound of $E[T_{MRCA}]$}

We need the following two lemmas about the asymmetric random walk on $\mathbb{N}_{0}:=\{0,1,\cdots\}$ with $0$ as a reflection boundary.
\begin{lemma}\label{rw}
Let $\{S_{n}\}_{n\in\mathbb{N}_{0}}$ be the asymmetric random walk on $\mathbb{N}_{0}$ reflected at $0$ such that $S_{n+1}=S_{n}+X_{n}$, where $X_{n}\sim Ber(p)$ is a Bernoulli random variable for $S_{n}\neq 0$ and $p\neq\frac{1}{2}$; $X_{n}\equiv 1$ for $S_{n}=0$, and $\{X_{n}\}_{n\in\mathbb{N}_{0}}$ are independent. Denote by $E[T(j, m)]$ the expected hitting time from $j$ to $m$ for $0\leq j<m$, then we have
\[E[T(j,m)]=\frac{m-j}{p-q}+\frac{2pq}{(p-q)^{2}}((\frac{q}{p})^{m}-(\frac{q}{p})^{j}),\]
where $q=1-p$.

In particular, we have
\begin{eqnarray*}
E[T(0,m)]&=&\frac{m}{p-q}+\frac{2pq}{(p-q)^{2}}((\frac{q}{p})^{m}-1),\\
E[T(m-1,m)]&=&\frac{1}{p-q}+\frac{2pq}{(p-q)^{2}}((\frac{q}{p})^{m}-(\frac{q}{p})^{m-1}).
\end{eqnarray*}
\end{lemma}
\begin{proof}
Following the method in Chapter XIV of \cite{feller1991introduction}, we are required to solve the following second order recurrence relation:
\begin{equation*}
\left\{\begin{array}{l}
E[T(j,m)]=1+qE[T(j-1,m)]+pE[T(j+1,m)], 1\leq j\leq m-1,\\
E[T(0,m)]=1+E[T(1,m)] \text{ and } E[T(m,m)]=0. 
\end{array}\right.
\end{equation*}	
The solution is of the form 
\[E[T(j,m)]=\frac{j}{q-p}+A+B(\frac{q}{p})^{j}, 0\leq j\leq m,\]
then by the boundary conditions
\[0=\frac{m}{q-p}+A+B(\frac{q}{p})^{m}, E[T(0,m)]=A+B, \text{ and } E[T_{1}]=\frac{1}{q-p}+A+B(\frac{q}{p}),\]
we have
\[A=\frac{m}{p-q}+\frac{2pq}{(p-q)^{2}}(\frac{q}{p})^{m},~B=\frac{-2pq}{(p-q)^{2}},\]
hence
\[E[T(j,m)]=\frac{m-j}{p-q}+\frac{2pq}{(p-q)^{2}}((\frac{q}{p})^{m}-(\frac{q}{p})^{j}).\]
\end{proof}

The following comparison lemma is intuitive, but we still provide a rigorous proof for the case when $p>\frac{1}{2}$.
\begin{lemma}\label{lem:crw}
Let $\{S_{n}\}_{n\in\mathbb{N}_{0}}$ and $\{S^{\prime}_{n}\}_{n\in\mathbb{N}_{0}}$ be two asymmetric random walk on $\mathbb{N}_{0}$ reflected at $0$ defined in Lemma \ref{rw}. For $\{S_{n}\}_{n\in\mathbb{N}_{0}}$, each step $X_{n}\sim Ber(p_{n})$ for $S_{n}\neq 0$ and $p_{n}\geq p>\frac{1}{2}$; but for $\{S^{\prime}_{n}\}_{n\in\mathbb{N}_{0}}$, each step $X^{\prime}_{n}\sim Ber(p)$ when $S_{n}\neq 0$. Then, for fixed $m\in\mathbb{N}$, we have
\[E[T(0,m)]\leq E[T^{\prime}(0,m)],\]
where $T(0,m)$ and $T^{\prime}(0,m)$ are hitting times from $0$ to $m$ for $\{S_{n}\}_{n\in\mathbb{N}_{0}}$ and $\{S^{\prime}_{n}\}_{n\in\mathbb{N}_{0}}$, respectively.
\end{lemma}

\begin{proof}
We start from $\{S_{n}\}_{n\in\mathbb{N}_{0}}$, and decrease the success (moving to right) probability at each position to $p$ one by one and backwardly. That is,
we first decrease the success probability at position $m-1$ and keep others unchanged. Note that $E[T(0,m)]=E[T(0,m-1)]+E[T(m-1,m)]$. In the subsequent, $T$ is gradually changing but we always use the same notation for simplicity.
Now, $E[T(0,m-1)]$ is not affected, but we claim that $E[T(m-1,m)]$ is increased. Actually, 
$P(T(m-1,m)=1)$ has been decreased from $p_{m-1}$ to $p$, but each $P(T(m-1,m)=k)$ for $k=3,5,\cdots$ has been increased: suppose that there are $l\geq 1$ steps from $m-1$ to $m-2$, which contributes a factor $q^{l}_{m-1}p_{m-1}$ to the probability $P(T(m-1,m)=k)$. By $p_{m-1}\geq p\geq \frac{1}{2}$, this factor is increased as $p_{m-1}$ is replaced by $p$ while others are not affected, thus each $P(T(m-1,m)=k)$ is increased. Consequently, we know that  
$E[T(m-1,m)]$ is increased.

Then, we continue to decrease $p_{m-2}$ to $p$. Since $E[T(0,m)]=E[T(0,m-2)]+E[T(m-2,m)]$, we need to argue that $E[T(m-2,m)]$ has been increased as $E[T_{0,m-2}]$ is not affected. Note that $E[T(m-2,m)]=E[T(m-2,m-1)]+E[T(m-1,m)]$, and now  $E[T(m-1,m)]=1+qE[T(m-2,m)]$ for $q=1-p$ as $p_{m-1}$ is replaced by $p$ in the last step. Therefore, $E[T(m-2,m)]=\frac{qE[T(m-2,m-1)]+1}{p}$. By the same argument for $E[T(m-1,m)]$, we know that $E[T_{m-2,m-1}]$ has been increased, and hence $E[T_{m-2,m}]$ has also been increased. 

Inductively, for $k\geq 2$, if we have decreased $p_{i}$ for $m-k+1\leq i\leq m-1$ to $p$, and now we replace $p_{m-k}$ by $p$. Since $E[T(0,m)]=E[T(0,m-k)]+E[T(m-k,m)]$, we need to argue that $E[T(m-k,m)]$ has been increased as $E[T(0,m-k)]$ is not affected. Same as $E[T(m-1,m)]$ and $E[T(m-2,m-1)]$ discussed above, we know that $E[T_{m-k, m-k+1}]$ has been increased.

We have the recurrence relation 
\[E[T(m-j,m)]=1+p_{m-j}E[T(m-j+1,m)]+q_{m-j}E[T(m-j-1,m)], 1\leq j\leq k,\]
where $p_{m-j}=p$ for $1\leq j\leq k-1$. Solving it as in Lemma \ref{rw}, we have
\[E[T(m-k+1,m)]=\frac{k-1}{p-q}+\frac{(\frac{q}{p})^{m-k+1}-(\frac{q}{p})^{m}}{(\frac{q}{p})^{m-k}-(\frac{q}{p})^{m}}(E[T(m-k,m)]+\frac{k}{q-p}).\]
Then, by $E[T(m-k,m)]=E[T(m-k,m-k+1)]+E[T(m-k+1,m)]$, we have
\[E[T(m-k,m)]=\frac{(\frac{q}{p})^{m-k}-(\frac{q}{p})^{m}}{(\frac{q}{p})^{m-k}-(\frac{q}{p})^{m-k+1}}\frac{k-1}{p-q}+\frac{(\frac{q}{p})^{m-k+1}-(\frac{q}{p})^{m}}{(\frac{q}{p})^{m-k}-(\frac{q}{p})^{m-k+1}}(E[T(m-k,m-k+1)]+\frac{k}{q-p}).\]
Consequently, $E[T(m-k,m)]$ has been increased as $E[T(m-k,m-k+1)]$ has been increased and $\frac{q}{p}<1$.
\end{proof}

Let $\overline{T}_{MRCA}(m):=\inf\{t\geq 0:(\overline{N}_{t}^{(0,m)},\overline{M}_{t}^{(0,m)})=(1,0)\}$ be the $T_{MRCA}$ of a sample consisting only of dormant individuals. We have the following asymptotic upper bound of $E[\overline{T}_{MRCA}(m)]$.
 
\begin{lemma}\label{lem:main}
If the support of $\mu$ has an upper bound $\overline{\lambda}_{0}$ and a lower bound $\underline{\lambda}_{0}$, then we have
\[\limsup_{||m||_{TV}\rightarrow\infty}\frac{E[\overline{T}_{MRCA}(m)]}{\log ||m||_{TV}}\leq \frac{2}{\underline{\lambda}_{0}}.\]
\end{lemma}

\begin{proof}
We first define recursively the following stopping times. For simplicity, we abuse $m$ and $||m||_{TV}$ without distinguishing between them from now on.  
\begin{eqnarray*}
D_{m}&=&\inf\{t\geq 0: \overline{N}^{(0,m)}_{t}\geq \lceil m^{\alpha} \rceil\};\\
H_{m-1}&=&\inf\{t\geq 0: \overline{N}^{(0,m)}_{t}+||\overline{M}^{(0,m)}_{t}||_{TV}=m-1\};\\
D_{m-1}&=&\inf\{t>H_{m-1}: \overline{N}^{(0,m)}_{t}\geq \lceil(m-1)^{\alpha}\rceil\};\\
H_{m-2}&=&\inf\{t\geq 0: \overline{N}^{(0,m)}_{t}+||\overline{M}^{(0,m)}_{t}||_{TV}=m-2\};\\
&\vdots&
\end{eqnarray*}
We have $D_{m}<H_{m-1}<D_{m-1}<H_{m-1}<\cdots<D_{2}<H_{1}=\overline{T}_{MRCA}(m)$
, and thus for $m_{0}\geq 2$,
\begin{equation}\label{sum}
E[\overline{T}_{MRCA}(m)]=E[D_{m}]+\sum_{j=m_{0}}^{m}E[H_{j-1}-D_{j}]+\sum_{j=m_{0}}^{m-1}E[D_{j}-H_{j}]+E[H_{1}-H_{m_{0}-1}].	
\end{equation}
For the convenience of the readers, we will use similar notations as in the proof of Lemma 4.11 in \cite{blath2016}, and we omit the superscript $(0,m)$ from now on. Let $\overline{\lambda}_{t}$ be the total jump rate of the process $(\overline{N}_{t}, \overline{M}_{t})$ at time $t\geq 0$. Assume that $\overline{M}_{t}=\sum\limits_{i=1}^{K_{t}}M_{i,t}\delta_{\Lambda_{i,t}}$ and $\overline{N}_{t}\geq 2$, then we have 
\[\overline{\lambda}_{t}=C^{2}_{\overline{N}_{t}}I_{\{\overline{N}_{t}\geq \lceil(\overline{N}_{t}+||\overline{M}_{t}||_{TV})^{\alpha}\rceil\}}+c\overline{N}_{t}+\sum_{i=1}^{K_{t}}\Lambda_{i,t}M_{i,t}.\]
Define
\[\overline{\alpha}_{t}:=\frac{C^{2}_{\overline{N}_{t}}I_{\{\overline{N}_{t}\geq \lceil(\overline{N}_{t}+||\overline{M}_{t}||_{TV})^{\alpha}\rceil\}}}{\overline{\lambda}_{t}}, \overline{\beta}_{t}:=\frac{c\overline{N}_{t}}{\overline{\lambda}_{t}},\text{ and }\overline{\gamma}_{t}:=\frac{\sum\limits_{i=1}^{K_{t}}\Lambda_{i,t}M_{i,t}}{\overline{\lambda}_{t}}\]
to be the probabilities of \textit{coalescence}, \textit{deactivation} and \textit{activation}, respectively.
The proof is divided into $4$ steps.

\textbf{Step 1}: Estimate $E[D_{m}]$ when $m$ is large enough.

The reason for defining $D_{m}$ is that the block counting process has to first attain the sate with $\lceil m^{\alpha} \rceil$ active blocks for a possible coalescence. For example, if $m=2$, then obviously we need to attain $(2,0)$ at first.

For any $t<D_{m}$, we have $\overline{N}_{t}\leq \lceil m^{\alpha}\rceil-1<m^{\alpha}$, and $\overline{\lambda}_{t}\geq \underline{\lambda}_{0}m$. Therefore, at the each jump time before $D_{m}$, we have
\[\overline{\beta}_{t}\leq \frac{c}{\underline{\lambda}_{0}m^{1-\alpha}} \text{ and }\overline{\gamma}_{t}=1-\overline{\beta}_{t}\geq 1-\frac{c}{\underline{\lambda}_{0}m^{1-\alpha}}.\]
We assume that $m>(\frac{2c}{\underline{\lambda}_{0}})^{\frac{1}{1-\alpha}}$ to make $\overline{\gamma}_{t}\geq \overline{\beta}_{t}$, and hence $1-\frac{c}{\underline{\lambda}_{0}m^{1-\alpha}}\geq \frac{1}{2}$. 
The expected number of jumps of the process until $D_{m}$ is therefore bounded above by the expected time $E[T(0,\lceil m^{\alpha}\rceil)]$ from $0$ to $\lceil m^{\alpha}\rceil$ for an asymmetric random walk on $\mathbb{N}_{0}$ with $0$ as a reflection boundary. By Lemma \ref{lem:crw} and Lemma \ref{rw}, we know that 
\[E[T(0,\lceil m^{\alpha}\rceil)]\leq \frac{\lceil m^{\alpha}\rceil}{1-\frac{2c}{\underline{\lambda}_{0}m^{1-\alpha}}}\leq (1+\epsilon)m^{\alpha}\]
for any given $\epsilon>0$ and $m\geq (\frac{\underline{\lambda}_{0}+2c(1+\epsilon)}{\underline{\lambda}_{0}\epsilon})^{\frac{1}{1-\alpha}}$. Since the time between two jumps of the process before $D_{m}$ is exponentially distributed with parameter $\overline{\lambda}_{t}$, then we have
\[E[D_{m}]\leq \frac{1}{\underline{\lambda}_{0}m}E[T(0, \lceil m^{\alpha}\rceil))]\leq \frac{1+\epsilon}{\underline{\lambda}_{0}m^{1-\alpha}}\]
for $m>(\frac{2c}{\underline{\lambda}_{0}})^{\frac{1}{1-\alpha}}\vee 
(\frac{\underline{\lambda}_{0}+2c(1+\epsilon)}{\underline{\lambda}_{0}\epsilon})^{\frac{1}{1-\alpha}}=(\frac{\underline{\lambda}_{0}+2c(1+\epsilon)}{\underline{\lambda}_{0}\epsilon})^{\frac{1}{1-\alpha}}$.

\textbf{Step 2}: Estimate $E[H_{j-1}-D_{j}]$ for $j\leq m$ when $j$ is large enough.

At the time $D_{j}$, we have $\overline{\lambda}_{t}\leq C^{2}_{\lceil j^{\alpha}\rceil}+(c\vee\overline{\lambda}_{0}) j$,
and 
\[\overline{\alpha}_{t}\geq \frac{C^{2}_{\lceil j^{\alpha}\rceil}}{C^{2}_{\lceil j^{\alpha}\rceil}+(c\vee\overline{\lambda}_{0})j}\geq \frac{j^{2\alpha}-j^{\alpha}}{j^{2\alpha}+j^{\alpha}+2(c\vee\overline{\lambda}_{0})j}\geq \frac{1}{1+\epsilon}\]
for any given $\epsilon>0$ and $j\geq (1+\frac{2+2(c\vee\overline{\lambda}_{0})}{                                                                                                                                                                                                                                                                                                                                                           \epsilon})^{\frac{1}{\alpha}}$. 

$\overline{\alpha}_{t}$ is the success probability of coalescence at the first jump time after $D_{j}$. If the process fails to coalesce, the number of active blocks will become $\lceil j^{\alpha}\rceil-1$ or $\lceil j^{\alpha}\rceil+1$. In the latter case, the requirement for a coalescence is still satisfied, and the expected time it takes for a jump is less than that of the form one if $j\geq \overline{\lambda}_{0}-c$. Therefore, we are going to consider the worser case, i.e., we always revert to the state with $\lceil j^{\alpha}\rceil-1$ active blocks if a coalescence does not occur at the first jump time after $D_{j}$. Since now the requirement for a coalescence is not satisfied, we need to attain the state with $\lceil j^{\alpha}\rceil$ active blocks again, and the expected number of steps is bounded above by the expected time $E[T(\lceil j^{\alpha}\rceil-1, \lceil j^{\alpha}\rceil)]$ from $\lceil j^{\alpha}\rceil-1$ to $\lceil j^{\alpha}\rceil$ for an asymmetric random walk on $\mathbb{N}_{0}$ with $0$ as a reflection boundary. 

As in Step 1, we have $\overline{\beta}_{t}\leq \frac{c}{\underline{\lambda}_{0}m^{1-\alpha}}$, and we assume that $j>(\frac{2c}{\underline{\lambda}_{0}})^{\frac{1}{1-\alpha}}$. Then, by Lemma \ref{lem:crw} and Lemma \ref{rw}, we know that 
\[E[T(\lceil j^{\alpha}\rceil-1,\lceil j^{\alpha}\rceil)]\leq \frac{1}{1-\frac{2c}{\underline{\lambda}_{0} j^{1-\alpha}}}\leq 1+\epsilon\]
for any given $\epsilon>0$ and $j\geq (\frac{\underline{\lambda}_{0}+2c(1+\epsilon)}{\underline{\lambda}_{0}\epsilon})^{\frac{1}{1-\alpha}}$. Since the 
process is not allowed to coalesce before it attains the state with $\lceil j^{\alpha}\rceil$ active blocks again, the expected time of each step from $\lceil j^{\alpha}\rceil-1$ to $\lceil j^{\alpha}\rceil$ is bounded above by $\frac{1}{\underline{\lambda}_{0}j}$. Taking the expected sojourn time at the state with $\lceil j^{\alpha}\rceil$ active blocks which is bounded above by $\frac{1}{C^{2}_{\lceil j^{\alpha}\rceil}+\underline{\lambda}_{0}j}$ into account, each independent Bernoulli trial for coalescence takes at most 
\[\frac{1+\epsilon}{\underline{\lambda}_{0}j}+\frac{1}{C^{2}_{\lceil j^{\alpha}\rceil}+\underline{\lambda}_{0}j}\leq \frac{1+2\epsilon}{\underline{\lambda}_{0}j}\]
for any given $\epsilon>0$ and $j>[1+(\frac{2}{\epsilon}-2)\underline{\lambda}_{0}]^{\frac{1}{2\alpha-1}}$ since $\alpha>\frac{1}{2}$.

In conclusion, the success probability of this independent Bernoulli trial for coalescence is bounded below by $\frac{1}{\epsilon}$, and the expected time it takes for each trial is bounded above by $\frac{1+2\epsilon}{\underline{\lambda}_{0}j}$, for any given $\epsilon>0$ and 
$j>(\frac{\underline{\lambda}_{0}+2c(1+\epsilon)}{\underline{\lambda}_{0}\epsilon})^{\frac{1}{1-\alpha}}\vee (\overline{\lambda}_{0}-c) 
\vee [1+(\frac{2}{\epsilon}-2)\underline{\lambda}_{0}]^{\frac{1}{2\alpha-1}}$.
Consequently, by Lemma \ref{lem:geo}, we have
\[E[H_{j-1}-D_{j}]\leq \frac{(1+\epsilon)(1+2\epsilon)}{\underline{\lambda}_{0}j}.\]

\textbf{Step 3}: Estimate $E[D_{j}-H_{j}]$ for $j\leq m-1$ when $j$ is large enough.

At the instant $H_{j}-$, there is a coalescence so that the total number of blocks decreases from $j+1$ to $j$, and by the restriction, we know that $\overline{N}_{H_{j}-}\geq \lceil (j+1)^{\alpha}\rceil$ and thus $\overline{N}_{H_{j}}\geq \lceil (j+1)^{\alpha}\rceil-1$. If $\overline{N}_{H_{j}}\geq \lceil j^{\alpha}\rceil$, then $D_{j}=H_{j}$. The only possible exception is the case that $\overline{N}_{H_{j}}=\lceil j^{\alpha}\rceil-1$ since $\lceil (j+1)^{\alpha}\rceil\geq\lceil j^{\alpha}\rceil$. In this worser case, the process is not allowed to coalesce during $(H_{j}, D_{j})$. Therefore, we have
\[E[D_{j}-H_{j}]=E[D_{j}-H_{j}|\overline{N}_{H_{j}}=\lceil j^{\alpha}\rceil-1]P(\overline{N}_{H_{j}}=\lceil j^{\alpha}\rceil-1)\leq E[D_{j}-H_{j}|\overline{N}_{H_{j}}=\lceil j^{\alpha}\rceil-1].\]

Same as in Step 2, the expected number of jumps to attain the state with $\lceil j^{\alpha}\rceil$ active blocks again is bounded above by $E[T(\lceil \sqrt{j}\rceil-1, \lceil \sqrt{j}\rceil)]$, and the expected time of each jump is bounded above by $\frac{1}{\underline{\lambda}_{0}j}$, thus we have
\[E[D_{j}-H_{j}]\leq \frac{1+\epsilon}{\underline{\lambda}_{0}j}\] 
for any given $\epsilon>0$ and $j>(\frac{\underline{\lambda}_{0}+2c(1+\epsilon)}{\underline{\lambda}_{0}\epsilon})^{\frac{1}{1-\alpha}}$.

\textbf{Step 4}: In the above 3 steps, we require $m$ or $j$ to be large enough, 
and note that the lower bounds are all constants only depending on $c,\underline{\lambda}_{0},\overline{\lambda}_{0}$ and any given $\epsilon>0$.
Putting them together, if  
\[m\geq j>(\frac{\underline{\lambda}_{0}+2c(1+\epsilon)}{\underline{\lambda}_{0}\epsilon})^{\frac{1}{1-\alpha}}\vee (\overline{\lambda}_{0}-c) 
\vee [1+(\frac{2}{\epsilon}-2)\underline{\lambda}_{0}]^{\frac{1}{2\alpha-1}},\]
then we have
\[E[D_{m}]\leq \frac{1+\epsilon}{\underline{\lambda}_{0} m^{1-\alpha}},~E[H_{j-1}-D_{j}]\leq \frac{(1+\epsilon)(1+2\epsilon)}{\underline{\lambda}_{0}j},\text{ and }E[D_{j}-H_{j}]\leq \frac{1+\epsilon}{\underline{\lambda}_{0}j}.\]

Now, in the definition of the \textit{decelerated coalescent} and (\ref{sum}), we take the threshold as
\[m_{0}=\lceil (\frac{\underline{\lambda}_{0}+2c(1+\epsilon)}{\underline{\lambda}_{0}\epsilon})^{\frac{1}{1-\alpha}}\vee (\overline{\lambda}_{0}-c) 
\vee [1+(\frac{2}{\epsilon}-2)\underline{\lambda}_{0}]^{\frac{1}{2\alpha-1}}\rceil,\]
then for $||m||_{TV}>m_{0}$, we have
\[E[\overline{T}_{MRCA}(m)]\leq \frac{1+\epsilon}{\underline{\lambda}_{0}m^{1-\alpha}}+\sum_{j=m_{0}}^{m}\frac{(1+\epsilon)(1+2\epsilon)}{\underline{\lambda}_{0}j}+\sum_{j=m_{0}}^{m-1}\frac{1+\epsilon}{\underline{\lambda}_{0}j}+E[H_{1}-H_{m_{0}-1}].\]

For $E[H_{1}-H_{m_{0}-1}]$, by the proof of Theorem \ref{th:mean}, Lemma \ref{lem:compare} and Remark \ref{rem:explicit}, we have a rough upper bound
\[E[H_{1}-H_{m_{0}-1}]\leq C^{2}_{m_{0}-1}E[T^{\{c\delta_{\underline{\lambda}_{0}}\}}_{MRCA}(0,2\delta_{\underline{\lambda}_{0}})]\leq C^{2}_{m_{0}-1}(1+\frac{4c+3}{2\underline{\lambda}_{0}}+\frac{2c^{2}+c}{2\underline{\lambda}_{0}^{2}}),\]
where the superscript $c\delta_{\underline{\lambda}_{0}}$ indicates the single seed-bank case.

Finally, we have
\begin{eqnarray*}
E[\overline{T}_{MRCA}(m)]&\leq & \frac{1+\epsilon}{\underline{\lambda}_{0}m^{1-\alpha}}+\sum_{j=m_{0}}^{m}\frac{(1+\epsilon)(1+2\epsilon)}{\underline{\lambda}_{0}j}+\sum_{j=m_{0}}^{m-1}\frac{1+\epsilon}{\underline{\lambda}_{0}j}+C^{2}_{m_{0}-1}(1+\frac{4c+3}{2\underline{\lambda}_{0}}+\frac{2c^{2}+c}{2\underline{\lambda}_{0}^{2}})\\
&\leq & \frac{2(1+\epsilon)^{2}}{\underline{\lambda}_{0}}\log m
\end{eqnarray*}
for any given $\epsilon>0$ and $m\geq m(c,\underline{\lambda}_{0},\overline{\lambda}_{0}, \epsilon)$, where $m(c,\underline{\lambda}_{0},\overline{\lambda}_{0}, \epsilon)>0$ is a constant only depending on $c,\underline{\lambda}_{0},\overline{\lambda}_{0}$ and $\epsilon$.
Since for any $\epsilon>0$, we have 
\[\limsup_{m\rightarrow\infty}\frac{E[\overline{T}_{MRCA}(m)]}{\log m}\leq \frac{2(1+\epsilon)^{2}}{\underline{\lambda}_{0}}.\]
Letting $\epsilon\rightarrow 0$, the desired result then follows.
\end{proof}

\begin{proof}[\textbf{Proof of Theorem \ref{th:bounds}-2}]
Let $S$ be the first time that all individuals which has never entered dormancy have coalesced into one block. Since the number of these individuals is at most $n$, $E[S]$ is bounded above by the $E[T_{MRCA}(n)]$ of the Kingman coalescent which is $2(1-\frac{1}{n})$. Therefore, we have
\begin{eqnarray*}
E[T_{MRCA}(n,m)]\leq 2+E[T_{MRCA}(N_{S}^{\{n,m\}}, M_{S}^{\{n,m\}})] \leq  2+E[T_{MRCA}(0, N_{S}^{\{n,m\}}\delta_{\underline{\lambda}_{0}}+M_{S}^{\{n,m\}})],
\end{eqnarray*} 
where the last inequality comes from Lemma \ref{lem:compare} and an induction argument. Intuitively, if we transform all active blocks at time $S$ to the worst dormancy state, then their coalescence time will be increased. Moreover, the number of blocks at time $S$ is at most $A^{n}+||m||_{TV}+1$, where 
$A^{n}$ is defined in the proof of Theorem \ref{th:bounds}-1, and it is the total number of blocks that have ever entered dormancy from the active state, $||m||_{TV}$ is the initial number of dormant blocks, and $1$ refers to the block composed of all individuals which has never entered dormancy. Considering the extreme case that all dormant blocks are not allowed to become active again, then obviously 
\[N_{S}^{\{n,m\}}+||M_{s}^{\{n,m\}}||_{TV}\leq A^{n}+||m||_{TV}+1.\]

Now, if we add additional $B:=A^{n}+||m||_{TV}+1-(N_{S}^{\{n,m\}}+||M_{s}^{\{n,m\}}||_{TV})$ dormant blocks with label $\underline{\lambda}_{0}$ at time $S$, then the upper bound will be further increased, and by Lemma \ref{lem:upper}, we have 
\begin{eqnarray}\label{eq:111}
E[T_{MRCA}(n,m)]&\leq & 2+E[T_{MRCA}(0, (B+N_{S}^{\{n,m\}})\delta_{\underline{\lambda}_{0}}+M_{S}^{\{n,m\}}]\nonumber\\
&\leq &	2+E[\overline{T}_{MRCA}((B+N_{S}^{\{n,m\}})\delta_{\underline{\lambda}_{0}}+M_{S}^{\{n,m\}})].
\end{eqnarray}
For simplicity, we abuse $A^{n}+||m||_{TV}+1$ and $(B+N_{S}^{\{n,m\}})\delta_{\underline{\lambda}_{0}}+M_{S}^{\{n,m\}}$ without distinguish between them from now on.

For any given $\epsilon>0$, we further have
\begin{eqnarray}\label{eq:444}
E[\overline{T}_{MRCA}(A^{n}+||m||_{TV}+1)]
&=& E[\overline{T}_{MRCA}(A^{n}+||m||_{TV}+1)I_{\{A^{n}<(2c+\epsilon)\log  n\}}]\nonumber\\
&&+	E[\overline{T}_{MRCA}(A^{n}+||m||_{TV}+1)I_{\{A^{n}\geq(2c+\epsilon)\log  n\}}].
\end{eqnarray}

Then, by the proof of Lemma \ref{lem:main}, 
\begin{eqnarray}\label{eq:222}
&&E[\overline{T}_{MRCA}(A^{n}+||m||_{TV}+1)I_{\{A^{n}<(2c+\epsilon)\log  n\}}]\leq  E[\overline{T}_{MRCA}((2c+\epsilon)\log n+||m||_{TV}+1)I_{\{A^{n}<(2c+\epsilon)\log  n\}}]\nonumber\\
&&\leq  \frac{2(1+\epsilon)^{2}}{\underline{\lambda}_{0}} \log((2c+\epsilon)\log n+||m||_{TV}+1)P(A^{n}<(2c+\epsilon)\log  n)	
\end{eqnarray}
for large enough $n+||m||_{TV}$.

By (\ref{eq:Abound1}) which comes from Chebyshev's inequality, we have
\begin{eqnarray}\label{eq:333}
E[\overline{T}_{MRCA}(A^{n}+||m||_{TV}+1)I_{\{A^{n}\geq(2c+\epsilon)\log  n\}}]&\leq& \frac{2(1+\epsilon)^{2}}{\underline{\lambda}_{0}} \log((2c+\epsilon)\log n+||m||_{TV}+1)\frac{2c\log n}{\epsilon^{2}(\log n)^{2}}\nonumber\\
&\leq& \frac{4c(1+\epsilon)^{2}}{\epsilon^{2}\underline{\lambda}_{0}}\frac{\log((2c+\epsilon)\log n+||m||_{TV}+1)}{\log n}	
\end{eqnarray}
for large enough $n+||m||_{TV}$.

Finally, combing (\ref{eq:111}), (\ref{eq:444}), (\ref{eq:222}) and (\ref{eq:333}) together, and divided by $\log(\log n+\frac{||m||_{TV}}{2c})$ on both sides, we have
\begin{eqnarray*}
\frac{E[T_{MRCA}(n,m)]}{\log(\log n+\frac{||m||_{TV}}{2c})}&\leq& \frac{2(1+\epsilon)^{2}}{\underline{\lambda}_{0}}\frac{\log(\log n+\frac{||m||_{TV}+1}{2c})+\log(2c+\epsilon)}{\log(\log n+\frac{||m||_{TV}}{2c})}P(A^{n}<(2c+\epsilon)\log  n),\\
&&+\frac{4c(1+\epsilon)^{2}}{\epsilon^{2}\underline{\lambda}_{0}}\frac{\log(\log n+\frac{||m||_{TV}+1}{2c})+\log(2c+\epsilon)}{\log(\log n+\frac{||m||_{TV}}{2c})\cdot\log n}.
\end{eqnarray*}

Letting $n\rightarrow \infty$ and $||m||_{TV}\rightarrow \infty$, by (\ref{eq:Abound1}), we then have
\[\limsup_{\substack{n\rightarrow\infty\\||m||_{TV}\rightarrow\infty}}\frac{E[T_{MRCA}(n,m)]}{\log(\log n+\frac{||m||_{TV}}{2c})}\leq \frac{2(1+\epsilon)^{2}}{\underline{\lambda}_{0}}.\]
The desired result then follows by letting $\epsilon\rightarrow 0$.
\end{proof}

The asymptotic upper bound of $E[T_{MRCA}(n,m)]$ for the Gamma distribution is left for future work.

\bibliographystyle{alpha}  
\bibliography{references}

\begin{thebibliography}{BGCKWB20}

\bibitem[Ber29]{bernstein1929fonctions}
Serge Bernstein.
\newblock Sur les fonctions absolument monotones.
\newblock {\em Acta Mathematica}, 52(1):1--66, 1929.

\bibitem[Ber09]{berestycki2009recent}
Nathana{\"e}l Berestycki.
\newblock Recent progress in coalescent theory.
\newblock {\em arXiv preprint arXiv:0909.3985}, 2009.

\bibitem[BGCKWB16]{blath2016}
Jochen Blath, Adri{\'a}n Gonz{\'a}lez~Casanova, Noemi Kurt, and Maite
  Wilke-Berenguer.
\newblock A new coalescent for seed-bank models.
\newblock {\em The Annals of Applied Probability}, 26(2):857--891, 2016.

\bibitem[BGCKWB20]{blath2020seed}
Jochen Blath, Adri{\'a}n Gonz{\'a}lez~Casanova, Noemi Kurt, and Maite
  Wilke-Berenguer.
\newblock The seed bank coalescent with simultaneous switching.
\newblock 2020.

\bibitem[Cox62]{cox1962renewal}
David~Roxbee Cox.
\newblock Renewal theory.
\newblock {\em Chapman and Hall}, 1962.

\bibitem[Fel91]{feller1991introduction}
William Feller.
\newblock {\em An introduction to probability theory and its applications,
  Volume 1}.
\newblock John Wiley \& Sons, 1991.

\bibitem[GdHO22]{greven2022spatial}
Andreas Greven, Frank den Hollander, and Margriet Oomen.
\newblock Spatial populations with seed-bank: well-posedness, duality and
  equilibrium.
\newblock {\em Electronic Journal of Probability}, 27:1--88, 2022.

\bibitem[Jia23]{jiao2023wright}
Likai Jiao.
\newblock Wright-fisher diffusion with a continuum of seed-banks.
\newblock {\em arXiv preprint arXiv:2307.07110}, 2023.

\bibitem[Kad18]{kadets2018course}
Vladimir Kadets.
\newblock {\em A course in functional analysis and measure theory}.
\newblock Springer, 2018.

\bibitem[Kin82a]{kingman1982coalescent}
John~FC Kingman.
\newblock The coalescent.
\newblock {\em Stochastic processes and their applications}, 13(3):235--248,
  1982.

\bibitem[Kin82b]{kingman1982genealogy}
John~FC Kingman.
\newblock On the genealogy of large populations.
\newblock {\em Journal of applied probability}, 19(A):27--43, 1982.

\bibitem[LdHWBB21]{lennon2021principles}
Jay~T Lennon, Frank den Hollander, Maite Wilke-Berenguer, and Jochen Blath.
\newblock Principles of seed banks and the emergence of complexity from
  dormancy.
\newblock {\em Nature Communications}, 12(1):4807, 2021.

\bibitem[LJ11]{lennon2011microbial}
Jay~T Lennon and Stuart~E Jones.
\newblock Microbial seed banks: the ecological and evolutionary implications of
  dormancy.
\newblock {\em Nature reviews microbiology}, 9(2):119--130, 2011.

\bibitem[Pit99]{pitman1999coalescents}
Jim Pitman.
\newblock Coalescents with multiple collisions.
\newblock {\em Annals of Probability}, pages 1870--1902, 1999.

\bibitem[SL18]{shoemaker2018evolution}
William~R Shoemaker and Jay~T Lennon.
\newblock Evolution with a seed bank: the population genetic consequences of
  microbial dormancy.
\newblock {\em Evolutionary applications}, 11(1):60--75, 2018.

\end{thebibliography}

\end{document}